\documentclass[11pt]{article}
\usepackage[french]{babel}
\usepackage{latexsym}
\usepackage{amsmath,amsxtra,amssymb,latexsym, amscd,amsthm,}
 \usepackage[pdftex]{color,graphicx}
 \usepackage[all]{xy}
\usepackage{pgf,tikz}
\usepackage[T1]{fontenc}
\usepackage{cleveref}
\usepackage{amsfonts}
\usetikzlibrary{arrows}
\usepackage{graphicx}
\usepackage{caption}
\usepackage{subcaption}
\usepackage{graphpap}
\usepackage{rotating}
\usepackage{bm}
\theoremstyle{plain}
\newtheorem{corollary}{Corollary}[section]
\newtheorem{thm}{Theorem}
\newtheorem{lemma}{Lemma}[section]
\newtheorem{question}{Question}

\theoremstyle{definition}
\newtheorem{exmp}{Exemple}[section]
\newtheorem{prop}{Proposition}
\newtheorem{definition}{Definition}[section]
\newtheorem{remark}{Remark}[section]

\newcommand{\R}{\mathbb{R}}

\newcommand{\n}{\mathbb{N}}
\newcommand{\sg}{\Sigma}
\newcommand{\gm}{\Gamma}

\title{\sc{Curves on surfaces and surgeries}}
\author{Abdoul Karim SANE}
\date{ \small{UMPA-ENS Lyon, January 30th 2019.}}

\begin{document}
\renewcommand{\proofname}{Proof}
\renewcommand{\abstractname}{Abstract}
\renewcommand{\refname}{Bibliography}
\maketitle
\begin{abstract}
We study collections of curves in generic position on a closed oriented surface whose complement are disks. We define a surgery operation on the set of such collections and we prove that any two of them can be connected by a sequence of surgeries. 
\end{abstract}

\begin{section}{Introduction}
Curves on surfaces is a rich domain in low dimensional topology and natural questions emerge from it. One of them is the counting problem under topological constraints. For instance it is well known that on a genus ~$g$ surface, there are ~$\lfloor \frac{g}{2}\rfloor + 1$ homeomorphism classes of simple closed curves.

 A collection of curves $\gm$ on a genus $g$ oriented closed surface $\sg_g$ is \textit{filling} if its complement $\sg_g-\gm$ is a union of topological disks. Considering $\gm$ as a graph embedded in $\sg_g$ whose vertices are the intersection points of $\gm$ and whose edges are the arcs connecting intersection points, a filling collection is the same object as a \textit{map} ---an embedded graph whose complement consists of disks. If $\gm$ is in generic position, all vertices have degree $4$.
 
 If $\sg_g-\gm$ consists of exactly one disk, we say that the collection is \textit{one-faced}. In terms of maps one usually speaks of a \textit{unicellular map}.
 
The counting problem for (unicellular) maps has been first considered by W. Tutte in the planar case \cite{Tut1,Tut2,Tut3, Tut4}, followed by Lehman-Walsh ~\cite{Walsh} and Harer-Zagier ~\cite{Zag} in higher genus. They give a counting formula for the number of unicellular maps in a genus $g$ surface and that formula has been reproved by direct bijective methods by G. Chapuy ~\cite{Chap}.  An exact formula for \textit{rooted} unicellular maps (that is maps with one marked oriented edge) was provided by A. Goupil and G. Schaeffer ~\cite{Goup}. For the number of quadrivalent maps on a genus $g$ surface with one marked oriented edge, their formula greatly simplifies into the number ~$\frac{(4g-2)!}{2^{2g-1}g!}$. Since marking one oriented edge multiplies the number of such maps by at most the number of oriented edges, $8g-4$ in this case, the number of one-faced collections up to homeomorphism grows exponentially with the genus. 

In a more topological context, T. Aougag and S. Shinnyih studied \textit{minimally intersecting pairs}: those one-faced collections made of exactly two simple closed curves. They show that their number also grows exponentially with the genus. Recently we studied one-faced collections in order to find symmetric polytopes in ~$\R^4$ which are not (dual) unit ball of intersection norms ~\cite{Element}.  

In this article, we endow the set of one-faced collections on a genus $g$ surface $\sg_g$ (up to homeomorphism of $\sg_g$) with an additional graph structure.

Given a one-faced collection $\gm$ and a simple arc $\lambda$ connecting two edges ~$x$ and ~$y$, we obtain a new collection $\gm'$ by "cutting-open" $\gm$ along $\lambda$ (see Figure ~\ref{surgery}). When the final collection is one-faced, we speak of a \textit{surgery} and denote it by ~$\sigma_{x,y}(\gm)$. We then define the \textit{surgery graph $K_g$} as the graph whose vertices are homeomorphism classes of one-faced collections of curves on $\sg_g$ and whose edges are pairs of collections connected by surgery. 

For example, if $g=1$ there is only one one-faced collection, so $K_1$ consists of one isolated vertex. Using the Goupil-Schaeffer formula for $g=2$, one checks that there are exactly six one-faced collections on $\sg_2$. \footnote{The Goupil-Schaeffer Formula gives 45 marked collections, but every one-faced collection corresponds to 3, 6, or 12 marked collections depending on the number of symmetries of the collection.} By inductively trying all possible surgeries, one obtains those six one-faced collections.

\begin{figure}[htbp]
\begin{center}
\includegraphics[scale=0.1]{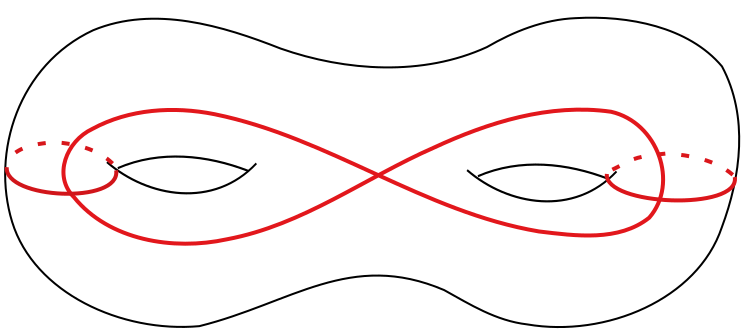}
\put(18,-30){\includegraphics[scale=0.1]{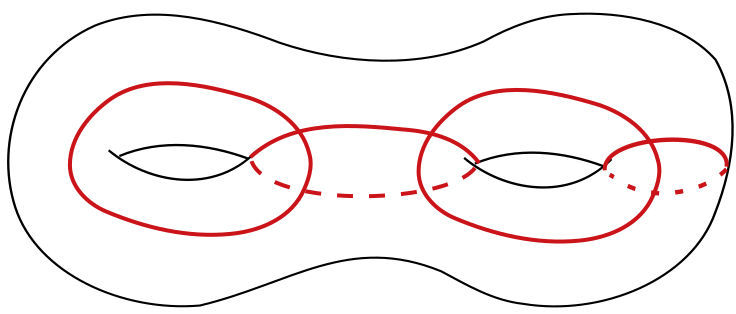}}
\put(18,30){\includegraphics[scale=0.1]{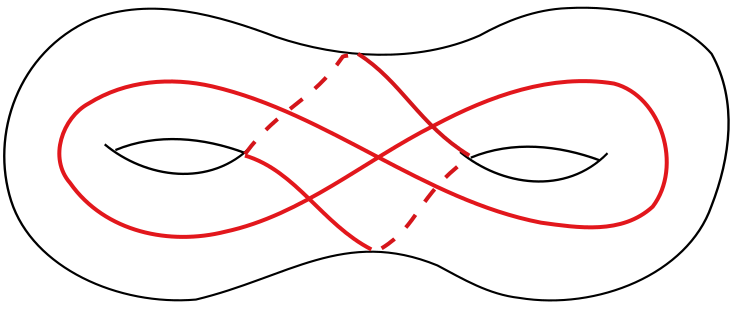}}
\put(-75,0){\includegraphics[scale=0.1]{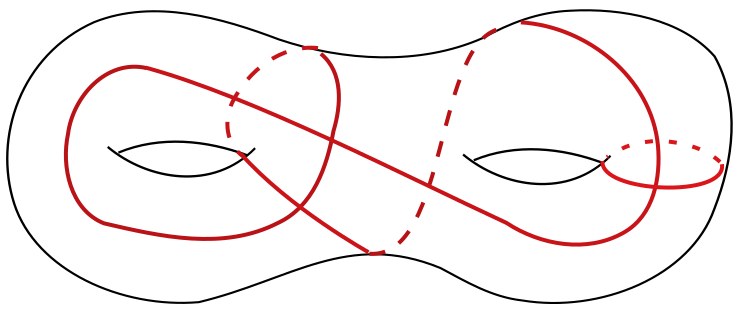}}
\put(-45,4){\huge{$\longleftrightarrow$}}
\put(7,18){\huge{$\nearrow$}}
\put(7,18){\huge{$\swarrow$}}
\put(7,-8){\huge{$\searrow$}}
\put(7,-8){\huge{$\nwarrow$}}
\put(-28,-30){\includegraphics[scale=0.1]{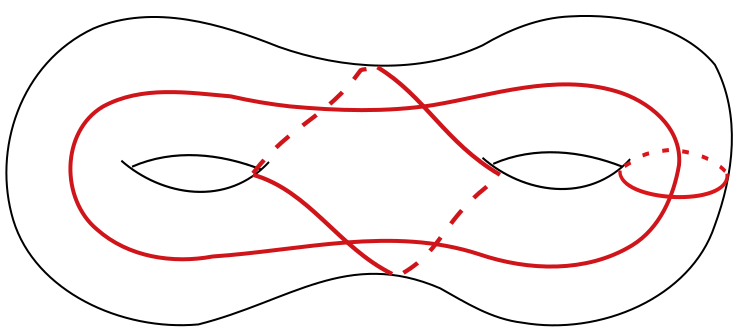}}
\put(-43,-15){\huge{$\nwarrow$}}
\put(-43,-15){\huge{$\searrow$}}
\put(-28,30){\includegraphics[scale=0.1]{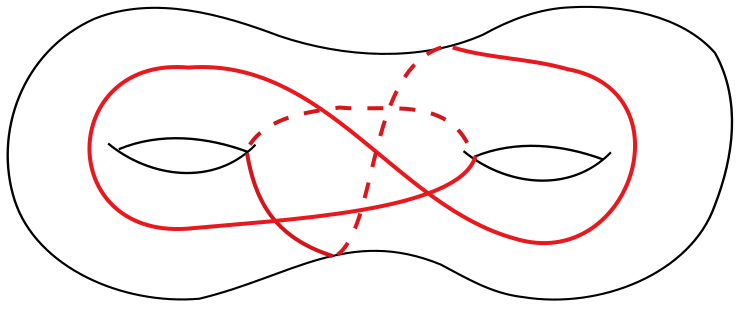}}
\put(-43,20){\huge{$\swarrow$}}
\put(-43,20){\huge{$\nearrow$}}
\put(-0,35){\huge{$\longleftrightarrow$}}
\put(-0,-25){\huge{$\longleftrightarrow$}}
\put(6,4){\huge{$-----$}}
\put(40,2){\includegraphics[scale=0.07]{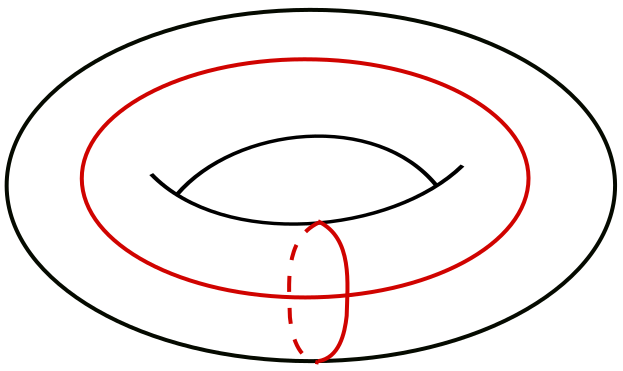}}
\caption{The graphs $K_2$ and $K_1$. The dashed edge indicates the vertex obtained by the connected sum of two copies of the unique collection in $K_1$.}
\label{Thegraph}
\end{center}
\end{figure}
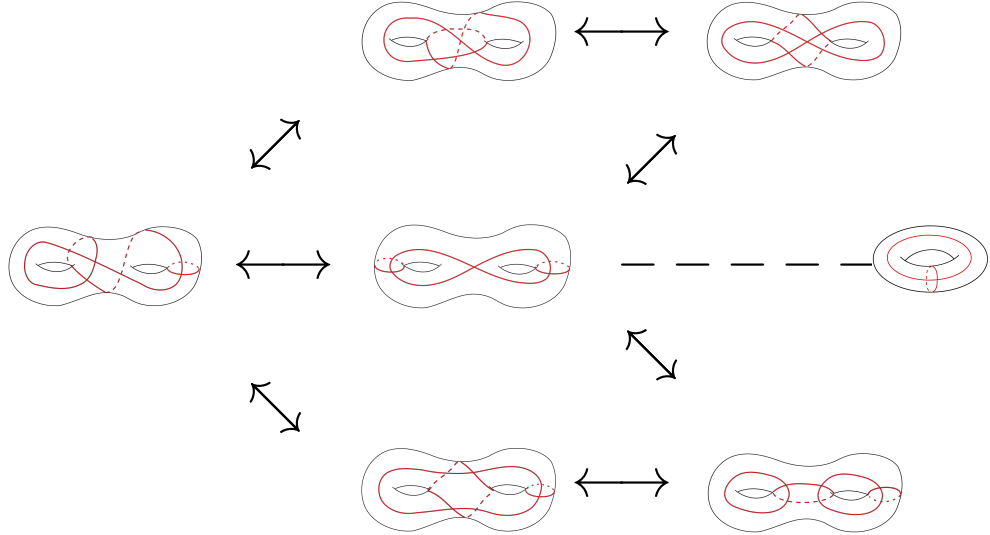

One notices that $K_2$ is connected (see Figure ~\ref{Thegraph}). Our main result is:
\begin{thm}\label{thm1}
For every integer $g$ the graph $K_g$ is connected.
\end{thm}
 
 Our proof is not straightforward as one may hope. We define a \textit{connected sum} operation on one-faced collections that turns two one-faced collections on $\sg_{g_1}$ and $\sg_{g_2}$ into a new one-faced collection on $\sg_{g_1+g_2}$. We then considers the \textit{surgery-sum graph} $\widehat{K}_g$ as the union $\sqcup_{i\leq g} K_i$ where one adds an edge between $\gm_1$ and $\gm_2$ if $\gm_2$ can be realized as a connected sum of $\gm_1$ with the unique one-faced collection on the torus. We then have: 
\begin{thm}\label{thm2}
For every $g$, the graph $\widehat{K}_g$ is connected.
\end{thm}
It is obvious that Theorem~\ref{thm1} implies Theorem~\ref{thm2}. 
However we were not able to find a direct proof of Theorem~\ref{thm1}, that is why we first prove Theorem~\ref{thm2} and then use the main ingredient of the proof (Proposition ~\ref{prop}) to prove Theorem~\ref{thm1}.

Now, denoting by $D_g$ the diameter of $K_g$, we also prove the following:
\begin{thm}\label{diam}
For every $g$, we have $D_g\leq3g^2+9g-12$.
\end{thm}

Let $K_{\infty}:=\displaystyle{\sqcup_{i\geq 1} \hat{K}_i}$ be the \emph{infinite surgery-sum graph} obtained by taking the union of all surgery-sum graphs. We also have:

\begin{thm}\label{hyper} The infinite surgery-sum graph $K_{\infty}$ is not Gromov hyperbolic. 
\end{thm}  
\begin{paragraph}{Organization of the article:} In Section ~\ref{sect2}, we recall some facts on one-faced collections.  In Section ~\ref{sect3}, we define the surgery operation on a one-faced collection and the connected sum of two one-faced collections. In Section ~\ref{sect4}, we prove a couple of lemmas which are going to be useful for the proof of Theorem ~\ref{thm1}, Theorem \ref{thm2} and Theorem \ref{diam} in section ~\ref{sect5}.
\end{paragraph}
\begin{paragraph}{Acknowledgments:}
I am thankful to my supervisors P. Dehornoy and J.-C Sikorav  for their support during this work. I am also grateful to Gregory Miermont for his indication to the works of G. Chapuy.
\end{paragraph}
\end{section}

\begin{section}{Combinatorial description of one-faced collections}\label{sect2}

In the whole paper, $\sg_g$ denotes an oriented closed surface of genus $g$. Let ~$\gm$ be a filling collection on $\sg_g$. One can consider $\gm$ as a graph embedded in ~$\sg_g$. Let $V$ be the number of vertices of $\gm$, $E$ the number of edges and $F$ the number of faces. The Euler characteristic of $\sg_g$ is given by $\chi(\sg_g)= 2-2g= V-E+F$.
As $\gm$ defines a regular  4-valent graph we have $E=2V$, which implies $V=2g-2+F.$ If $\gm$ is one-faced, we then have  
\[V=2g-1, \quad E=4g-2, \quad F=1.\] 
From this one can see that there is only one one-faced collection on the torus. We call it $\gm_{\mathbb{T}}$ (see Figure ~\ref{tor}).

\begin{paragraph}{Gluing pattern for a one-faced collection:} Let $\gm$ be a one-faced collection on $\sg_g$, then $\sg_g-\gm$ is a polygon with $8g-4$ edges, that we denote by $P_{\gm}$. The polygon $P_{\gm}$ comes with a pairwise identification of its edges. Choosing an edge on $P_{\gm}$ as an origin, one can label the edges of $P_{\gm}$ from the origin in a clockwise manner; thus obtaining a word $W_{\gm}$ on $8g-4$ letters. If two edges are identified, we label them with the same letter with a bar on the letter of the second edge. The word $W_{\gm}$ is a \textit{gluing pattern} of $(\sg_g,\gm)$ and two gluing patterns of $(\sg_g,\gm)$ differ by a cyclic permutation and a relabeling. A letter of a gluing pattern associated to a one-faced collection corresponds to a side of an edge of $\gm$; thus we can see a letter as an \textit{oriented edge}.

\begin{exmp} The word $W=ab\bar{a}\bar{b}$ is a gluing pattern for the one-faced collection $\gm_{\mathbb{T}}$ on the torus.
\begin{figure}[htbp]
\begin{center}
\includegraphics[scale=0.2]{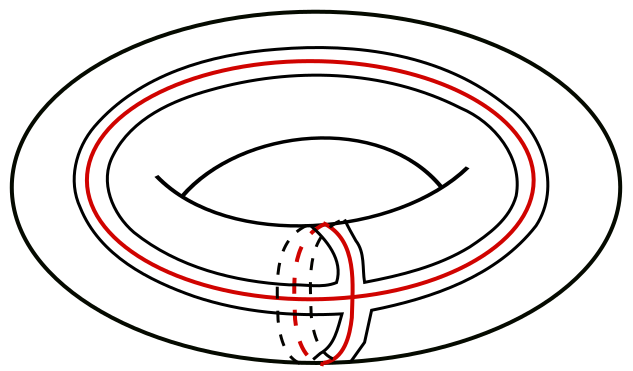}
\put(-22,23.4){\small{$a$}}
\put(-22,18){\small{$\bar{a}$}}
\put(-24,1){\small{$b$}}
\put(-18,1){\small{$\bar{b}$}}
\caption{The one-faced collection $\gm_{\mathbb{T}}$ on the torus.}
\label{tor}
\end{center}
\end{figure}
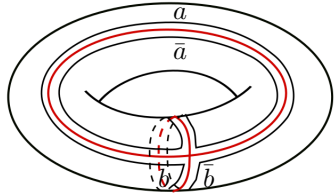
\end{exmp}     
We have the following:
\begin{prop}
\textit{Let $\gm_1$ and $\gm_2$ be two one-faced collections on $\sg_g$. Then ~$\gm_1$ and $\gm_2$ are topologically equivalent if and only if they have the same gluing patterns up to cyclic permutation an relabeling.} 
\end{prop}
\begin{proof} Let us assume that $\gm_1$ and $\gm_2$ have the same gluing patterns. Since $\sg_g-\gm_1$ and $\sg_g-\gm_2$ are disks, they are homeomorphic. The fact that $\gm_1$ and $\gm_2$ have the same gluing patterns implies that one can choose a homeomorphism $$\tilde{\phi}:P_{\gm_1}\longrightarrow P_{\gm_2}$$ such that $\tilde{\phi}$ maps a couple of identified sides on $P_{\gm_1}$ to a couple of identified sides on $P_{\gm_2}$. Therefore, $\tilde{\phi}$ factors to $\phi$ on the quotient (identification of sides) and $\phi(\gm_1)=\gm_2$.\vspace{0.3cm}\\
Conversely, if $\phi(\gm_1)=\gm_2$ then $\gm_1$ and $\gm_2$ have the same gluing pattern. 
\end{proof} 
\end{paragraph}

\begin{paragraph}{One-faced collections as permutations:}
Maps on surfaces can be described by a triple of permutations which satisfy some conditions (see \cite{Zvk} for the combinatorial definition of maps). We restrict that definition to the case of one-faced collection.

To a one-faced collection we associate $H$: the set of oriented edges, an involution $\alpha$ of $H$ which maps an oriented edge to the same edge with opposite orientation and a permutation $\mu$ whose cycles are oriented edges emanating from vertices when we turn counter-clockwise around them.   
\begin{definition}\label{combin}
The triple $(H,\alpha, \mu)$ is the combinatorial definition of a one-faced collection and $\gamma:=\alpha\mu$ describes the face of $\gm$.
\end{definition}
If $W_{\gm}$ is a gluing pattern of $\gm$, we can take $H$ to be the set of letters of $W_{\gm}$. The permutation $\gamma$ is then the shift to the right and it corresponds to the unique face. The cycles of $\alpha$ and $\mu$ correspond to the edges and vertices of $\gm$, respectively. Moreover, if we fix an origin $x\in H$  we get a natural order from $\gamma$:
$$x<\gamma(x)<......<\gamma^{8g-3}(x).$$
Changing the origin, the order above changes cyclically. 

The cycles of $\mu$ are in one to one correspondence with the vertices of $\gm$. (see Figure ~\ref{avertex}).

If $x$ and $y$ are two oriented edges with $x<y$, they define two intervals in a gluing pattern for $\gm$: 
\begin{exmp} The one-faced collection on the torus is given by the following permutations: \[\mu=(a \bar{b} \bar{a} b); \quad \alpha=(a \bar{a})(b \bar{b}); \quad \gamma=(a b \bar{a} \bar{b}).\]
\end{exmp} 
\end{paragraph}
\end{section}

\begin{section}{Surgery and connected sum on one-faced collections}\label{sect3}
In this section, we define two topological operations on the set of one-faced collections. 
\begin{paragraph}{Surgery on a one-faced collection:}
Let $\gm$ be a one-faced collection on ~$\sg_g$, $x$ and $y$ be two oriented edges of $\gm$ ($x$ and $y$ correspond to two sides of $P_{\gm}$). Since $\gm$ is one-faced, there is a unique homotopy class of simple arcs whose interiors are disjoint from $\gm$ and with endpoints in $x$ and $y$; let us denote it by ~$\lambda_{x,y}$. We obtain a new collection denoted by ~$\sigma_{x,y}(\gm)$ by "cutting-open" $\gm$ along $\lambda_{x,y}$ (see Figure ~\ref{surgery}). The collection $\sigma_{x,y}(\gm)$ is not necessarily one-faced.
\begin{figure}[htbp]
\begin{center}
\[
   \xymatrix{
\includegraphics[scale=0.10]{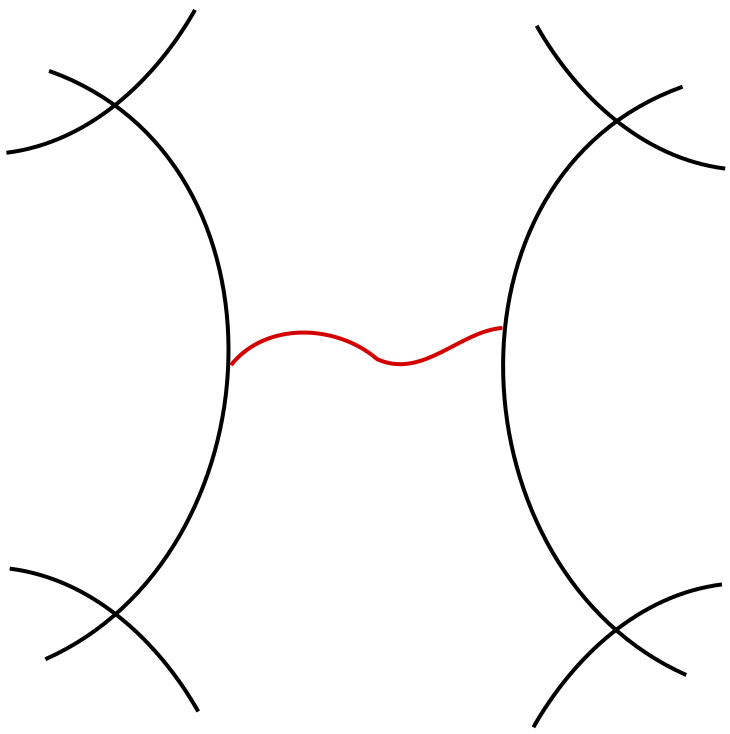}  &  &
\includegraphics[scale=0.10]{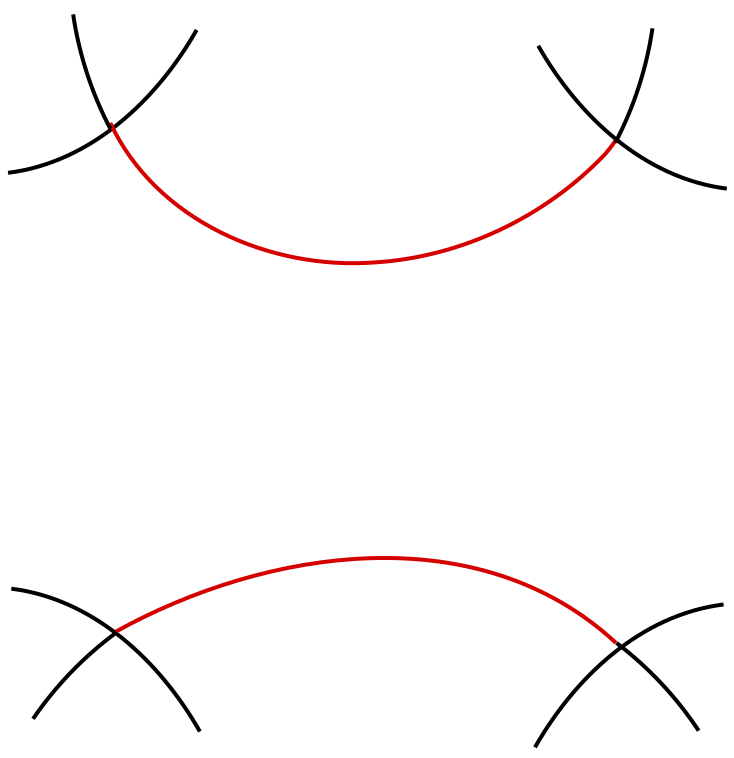}
\put(-65,18){\small{$x$}}
\put(-56,18){\small{$y$}}
\put(-62,10){\small{$\lambda_{x,y}$}}
}
\]
\caption{Surgery between two oriented edges $x$ and $y$ along ~$\lambda_{x,y}$. On the left-hand side, we have the arc $\lambda_{x,y}$ from $x$ to $y$ and on the right-hand side the cutting-open operation along ~$\lambda_{x,y}$.}
\label{surgery}
\end{center}
\end{figure}
 \begin{definition}
Let $\gm$ be a one-faced collection, $x$ and $y$ be to oriented edges of ~$\gm$. We say that $\{x,y\}$ and $\{\bar{x},\bar{y}\}$ are \textit{intertwined} if $\bar{x}$ and $\bar{y}$ are not both in $\overset{\frown}{[x,y]}$ and not both in $\overset{\smile}{[y,x]}$. It means that ~$\gm$ admits a gluing pattern of the form $w_1\bm{x}w_2\bm{\bar{x}}w_3\bm{y}w_4\bm{\bar{y}}$.

Otherwise, we say that $\{x,y\}$ and $\{\bar{x},\bar{y}\}$ are \textit{not intertwined}.
\noindent By abuse, we will just say that $x$ and $y$ are intertwined or not intertwined.
\end{definition}

Now, the following lemma gives a necessary and sufficient condition for the above operation to preserve the one-faced character.   
\begin{lemma}\label{lemsurg}
Let $\gm$ be a one-faced collection, $x$ and $y$ be two oriented edges of $\gm$. Then $\sigma_{x,y}(\gm)$ is one-faced if and only if $x$ and $y$ are intertwined. In this case, we call the operation a \rm{surgery} \textit{on ~$\gm$ between $x$ and $y$.} 

\noindent \textit{Moreover, if $w_1\bm{x}w_2\bm{\bar{x}}w_3\bm{y}w_4\bm{\bar{y}}$ is a gluing pattern for $\gm$ then, $$w_3\bm{X}w_2\bm{\bar{X}}w_1\bm{Y}w_4\bm{\bar{Y}}$$ is a gluing pattern for ~$\sigma_{x,y}(\gm)$}.
\end{lemma}

\begin{proof}
Since the operation along $\lambda_{x,y}$ leads to a new collection ~$\gm':=\sigma_{x,y}(\gm)$, all we have to do is to prove that $\gm'$ is one-faced. We use a cut and past argument similar to several proofs of the classification of surfaces (see Figure~\ref{cutpaste}).

Assume first that $x$ and $y$ are intertwined. When we "cut-open" along ~$\lambda_{x,y}$, the edges $\{x,\bar{x}\}$ and $\{y,\bar{y}\}$ get replaced by new edges $\{X,\bar{X}\}$ and $\{Y,\bar{Y}\}$. When we cut along the two new edges (in the polygonal description) and glue along the old ones (see Figure ~\ref{cutpaste}), we obtain a polygon; that is $\gm'$ is one-faced with gluing pattern $$w_3Xw_2\bar{X}w_1Yw_4\bar{Y}.$$

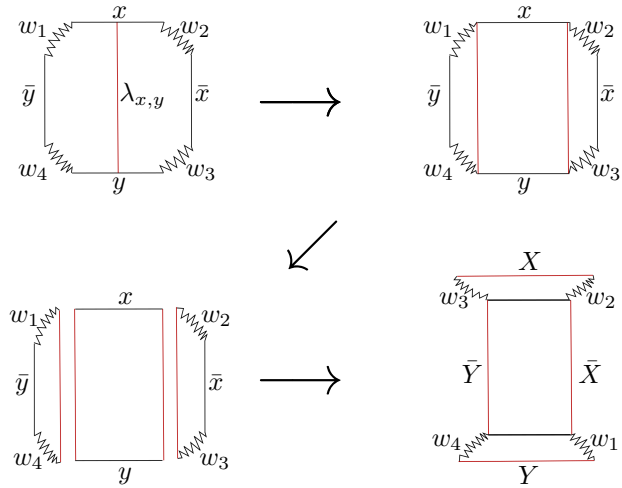
\begin{figure}[htbp]
\begin{center}
\[
   \xymatrix{
       \includegraphics[scale=0.1]{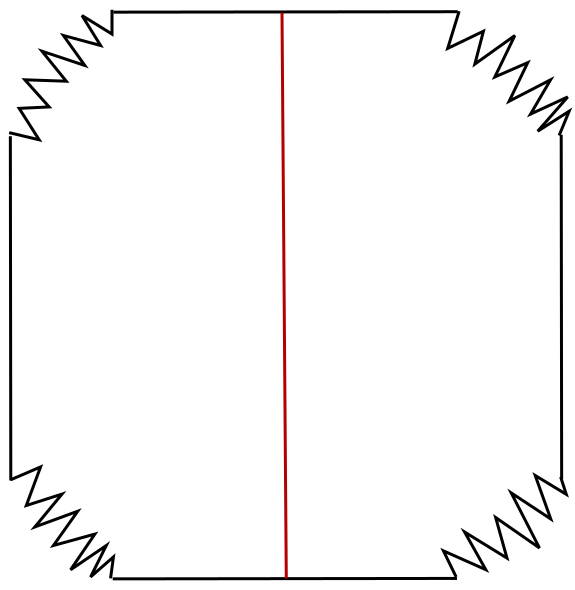} \put(-23.5,19){\small{$w_1$}} \put(-11,21){\small{$x$}} \put(-2,19){\small{$w_2$}} \put(0,10){\small{$\bar{x}$}} \put(-1.2,0){\small{$w_3$}} \put(-11,-1.8){\small{$y$}} \put(-23.5,0){\small{$w_4$}} \put(-23,10){\small{$\bar{y}$}} \put(-10,10){\small{$\lambda_{x,y}$}} & & & \includegraphics[scale=0.1]{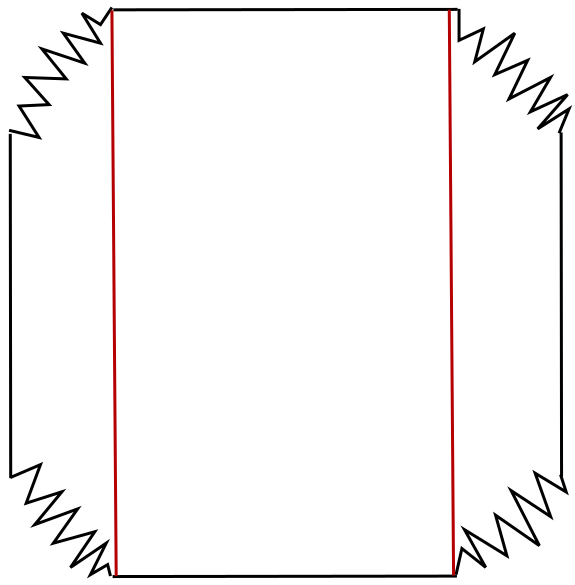} \put(-23.5,19){\small{$w_1$}} \put(-11,21){\small{$x$}} \put(-2,19){\small{$w_2$}} \put(0,10){\small{$\bar{x}$}} \put(-1.2,0){\small{$w_3$}} \put(-11,-1.8){\small{$y$}} \put(-23.5,0){\small{$w_4$}} \put(-23,10){\small{$\bar{y}$}}\\
      \includegraphics[scale=0.1]{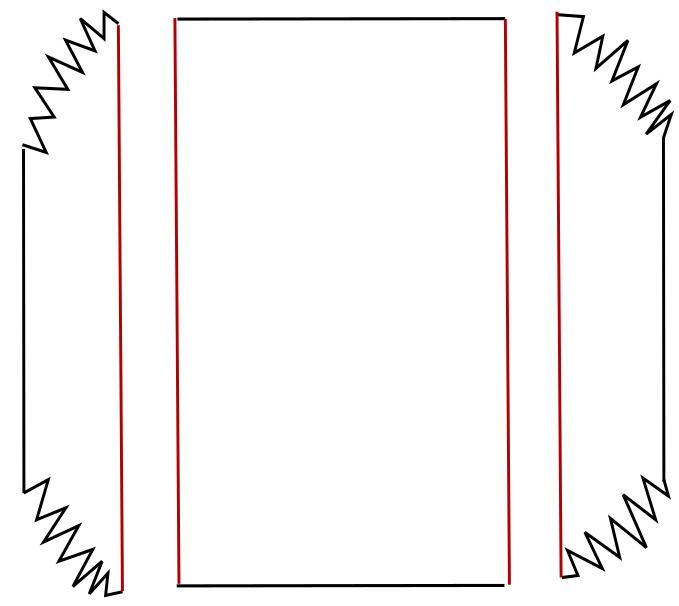} \put(-26.5,19){\small{$w_1$}} \put(-12,21.5){\small{$x$}} \put(-1,19){\small{$w_2$}} \put(0,10){\small{$\bar{x}$}} \put(-1.2,0){\small{$w_3$}} \put(-12,-1.8){\small{$y$}} \put(-26,0){\small{$w_4$}} \put(-25.7,10){\small{$\bar{y}$}}& && \includegraphics[scale=0.1]{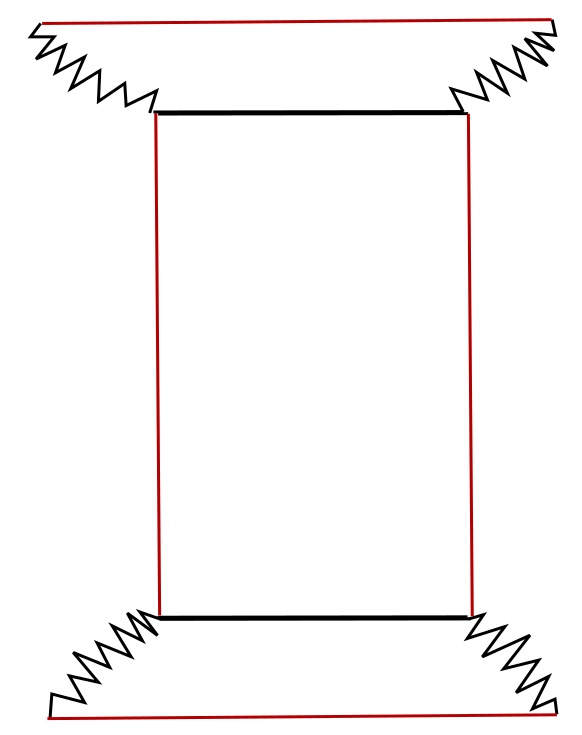}\put(-21.5,21.5){\small{$w_3$}} \put(-11,26){\small{$X$}} \put(-2,21.5){\small{$w_2$}} \put(-3,11){\small{$\bar{X}$}} \put(-1.5,2.5){\small{$w_1$}} \put(-11,-2.5){\small{$Y$}} \put(-22.5,2.5){\small{$w_4$}} \put(-18.5,11){\small{$\bar{Y}$}}
                 }
\put(-47,8){\huge{$\longrightarrow$}}
\put(-47,-29){\huge{$\longrightarrow$}}
\put(-47,8){\huge{$\longrightarrow$}}
\put(-43,-11){\huge{$\swarrow$}}
\]
\caption{Cut and paste on the polygon $P_{\gm}$.}
\label{cutpaste}
\end{center}
\end{figure} 

On the other hand, if $x$ and $y$ are not intertwined, one constructs an essential curve disjoint from $\gm'$, so $\gm'$ is not one-faced (see Figure ~\ref{degenerate}).

\begin{figure}[htbp]
\begin{center}
\includegraphics[scale=0.13]{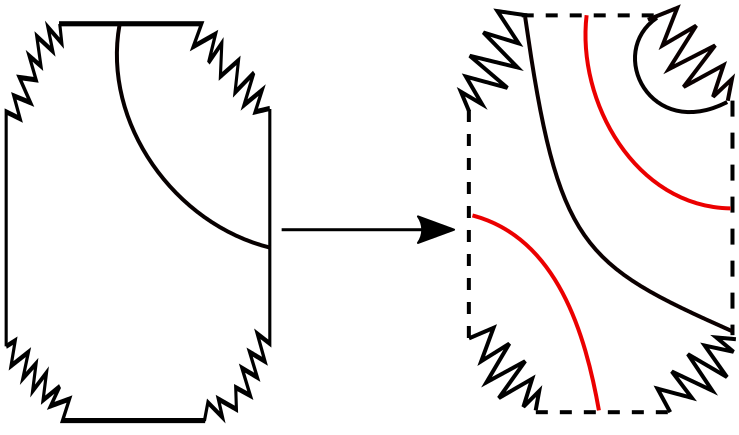}
\put(-29,20){\small{$x$}}
\put(-21,12){\small{$y$}}
\caption{The two arcs in red color define an essential closed curve on $\sg_g$ disjoint from $\gm'$ since it intersects $\gm$ algebraically twice.}
\label{degenerate}
\end{center}
\end{figure}

\end{proof}
\begin{remark} 
 The word $W_{\sigma_{x,y}(\gm)}$ in Lemma ~\ref{lemsurg} is obtained by permuting $w_1$ and $w_3$. It is also equivalent to permute $w_2$ and $w_4$.
 \end{remark}
 
 \begin{remark}
 If $x$ and $y$ are two intertwined  oriented edges, then $\bar{x}$ and $\bar{y}$ are also intertwined. Moreover one has $\sigma_{x,y}(\gm)=\sigma_{\bar{x},\bar{y}}(\gm)$. In fact, by Lemma ~\ref{lemsurg}, if $W_{\gm}=w_1xw_2\bar{x}w_3yw_4\bar{y}$ is a gluing pattern for $\gm$, $W_{\sigma_{\bar{x},\bar{y}}(\gm)}=w_1Xw_4\bar{X}w_3Yw_2\bar{Y}$ is a gluing pattern for $\sigma_{\bar{x},\bar{y}}(\gm)$, and it is equivalent to $W_{\sigma_{x,y}(\gm)}=w_3Xw_2\bar{X}w_1Yw_4\bar{Y}$ up to cyclic permutation and relabeling. 
 \end{remark}
 
 \begin{remark}
Given a one-faced collection, there are always intertwined  pairs unless it is the one-faced collection in the torus. Indeed, if all pairs of ~$\gm$ are not intertwined a gluing pattern for $\gm$ is given by $$W_{\gm}=x_1x_2.....x_{4g-2}\bar{x}_1\bar{x}_2....\bar{x}_{4g-2}.$$
After identifying the sides of $P_{\gm}$, all the vertices of $P_{\gm}$ get identified. Thus $\gm$ has only one self-intersection point. It follows that $g=1$ and that $\gm$ is the only one-faced collection with one self-intersection point, namely $\gm_{\mathbb{T}}$.  
\end{remark}
\end{paragraph}

\begin{paragraph}{Connected sum:} Let $\gm_1$ and $\gm_2$ be two one-faced collections on two surfaces $\sg_1$ and ~$\sg_2$, respectively. Let $D_1$ and $D_2$ be two open disks on $\sg_1$ and ~$\sg_2$, disjoint from ~$\gm_1$ and ~$\gm_2$, respectively. Let ~$\sg_{g_1}\#\sg_{g_2}$ be the connected sum along $D_1$ and $D_2$. Then ~$(\sg_{g_1}\#\sg_{g_2},\gm_1\cup\gm_2)$ is a genus $g_1+g_2$ surface endowed with a collection ~$\gm_1\cup\gm_2$. 
Since ~$\gm_1$ and $\gm_2$ are one-faced, the complement of ~$\gm_1\cup\gm_2$ in ~$\sg_{g_1}\#\sg_{g_2}$ is an annulus. 

Now, let $x$ and $y$ be two oriented edges of $\gm_1$ and $\gm_2$ respectively, and ~$\lambda_{x,y}$ a simple arc on $\sg_{g_1}\#\sg_{g_2}$ from $x$ to $y$ whose interior is disjoint from ~$\gm_1\cup\gm_2$. The arc $\lambda_{x,y}$ joins the two boundary components of $\sg_{g_1}\#\sg_{g_2}-{\gm_1\cup\gm_2}$. Therefore the graph $\gm_1\cup\gm_2\cup\lambda_{x,y}$ fills $\sg_{g_1}\#\sg_{g_2}$ with one disk in its complement.

Thus, the collection $\gm':=(\gm_1\cup\gm_2\cup\lambda_{x,y})/\lambda_{x,y}$ ---the quotient here means the contraction of $\lambda_{x,y}$ to a point--- (see Figure ~\ref{sumword}) is a one-faced collection. We say that $\gm'$ is the \textit{connected sum} of the marked collections $(\gm_1,x)$ and ~$(\gm_2,y)$.

\begin{figure}[htbp]
\begin{center}
\[
   \xymatrix{
 \includegraphics[scale=0.1]{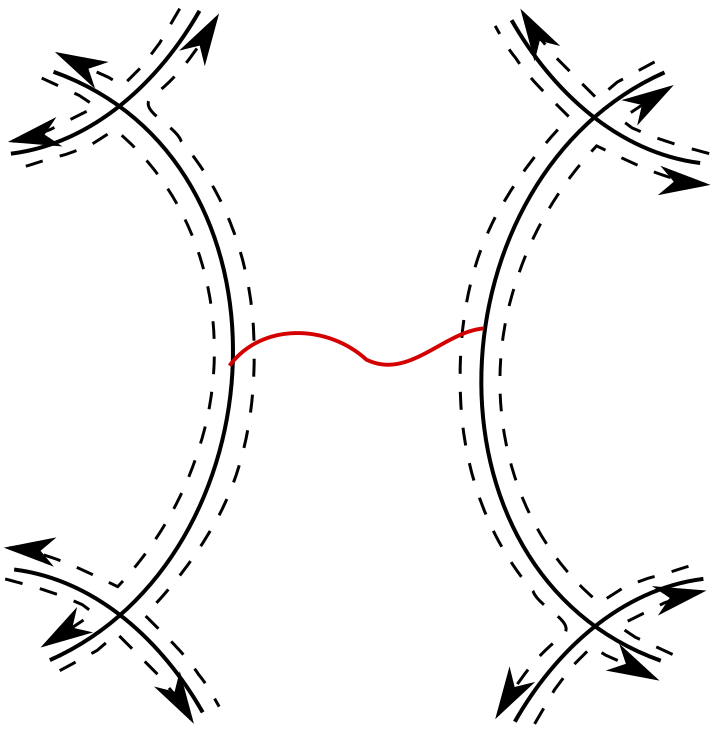}&  &\includegraphics[scale=0.1]{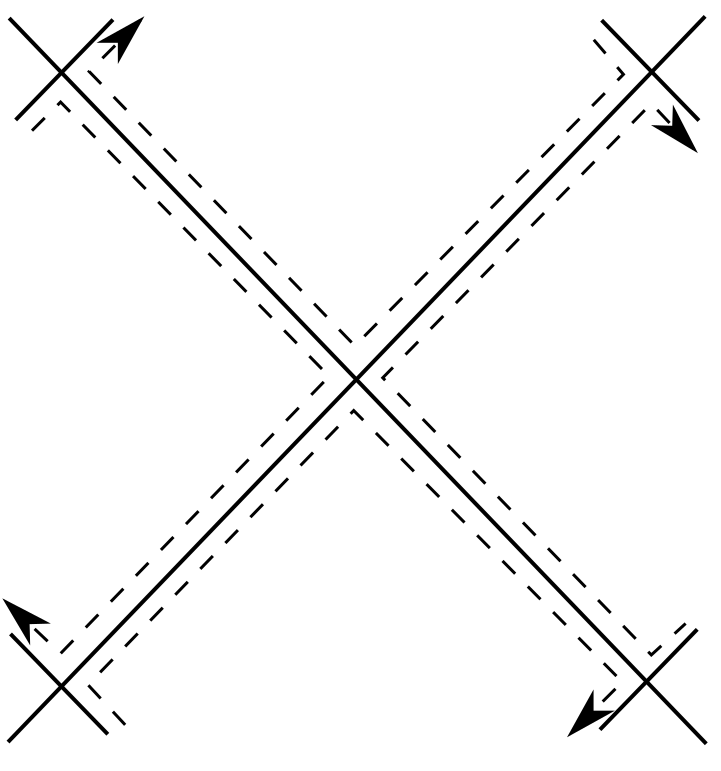}
\put(-64,18){\small{$x$}}
\put(-68,16.5){\small{$\bar{x}$}}
\put(-56.5,18){\small{$y$}}
\put(-52.5,16.5){\small{$\bar{y}$}}
\put(-62,10){\small{$\lambda_{x,y}$}}
\put(-19,21.8){\small{$x_1$}}
\put(-22.5,18){\small{$\bar{x}_1$}}
\put(-22.5,8){\small{$x_2$}}
\put(-19,5){\small{$\bar{x}_2$}}
\put(-10,5){\small{$y_1$}}
\put(-5.5,8){\small{$\bar{y}_1$}}
\put(-5.5,18){\small{$y_2$}}
\put(-10,21){\small{$\bar{y}_2$}}
}
\]
\caption{}
\label{sumword}
\end{center}
\end{figure}

\begin{definition}
Let $\gm$ be a one-faced collection, $x$ and $y$ be two oriented edges of $\gm$. We say that $x$ and $y$ are \textit{symmetric} if the gluing pattern for $\gm$ starting at $x$ is the same as the one starting at $y$ up to relabeling. 
\end{definition}
\begin{exmp}
On $\gm_{\mathbb{T}}$, any two oriented edges are symmetric.  
\end{exmp}

The following lemma states how we obtain a gluing pattern for $(\gm_1,x)\#(\gm_2,y)$ from gluing patterns for $\gm_1$ and $\gm_2$. 

\begin{lemma}\label{sum}
If $W_{\gm_1}=\bm{x}w_1\bm{\bar{x}}w_2$ (respectively $W_{\gm_2}=\bm{y}w'_1\bm{\bar{y}}w'_2$) is a gluing pattern for $\gm_1$ (respectively $\gm_2$), then  \[\bm{x_1}w_1\bm{\bar{x}_1x_2}w_2\bm{\bar{x}_2}\bm{y_1}w'_1\bm{\bar{y}_1}\bm{y_2}w'_2\bm{\bar{y}_2}\] is a gluing pattern for $(\gm_1,x)\#(\gm_2,y)$.

Moreover, $(\gm_1,x)\#(\gm_2,y)$ and $(\gm_1,x)\#(\gm_2,y')$ are topologically equivalent if $y$ and $y'$ are symmetric.    
\end{lemma}
\begin{proof}
The proof can be read on Figure ~\ref{sumword}. 
\end{proof}
Lemma \ref{sum} implies that the connected sum of a one-faced collection $\gm$ with $\gm_{\mathbb{T}}$ depend only on the oriented edge we choose on $\gm$, since all oriented edges of $\gm_{\mathbb{T}}$ are symmetric.
\end{paragraph}
\end{section}

\begin{section}{Surgery classification of one-faced collections}\label{sect4}
In this section, we study how surgeries connects one-faced collections. 
\begin{paragraph}{Simplification of one-faced collections:} Let $\gm$ be a one-faced collection, $(H, \alpha, \mu, \gamma)$ the permutations associated to $\gm$ (Definition \ref{combin}) and $x$ an oriented edge of $\gm$. Then the oriented edges $x$ and $C(x):=\gamma\alpha\gamma(x)$ belong to the same curve $\beta\in\gm$; $x$ and $C(x)$ are consecutive along $\beta$ (see Figure ~\ref{cons}). Moreover, the sequence $(C^n(x))_n$ is periodic and it travels through all edges of $\beta$.
\begin{definition}
Let $\gm$ be a one-faced collection and $\theta\in\gm$ a simple curve. The curve $\theta$ is \textit{1-simple} if $\theta$ intersects exactly one time $\gm-\theta$.
\end{definition}

Note that we have $C(x)=x$ if and only if $x$ is a side of a 1-simple curve ~$\theta\in\gm$. We denote by $S_{\gm}$ the number of 1-simple closed curves in $\gm$.
\begin{figure}[htbp]
\begin{center}
\includegraphics[scale=0.18]{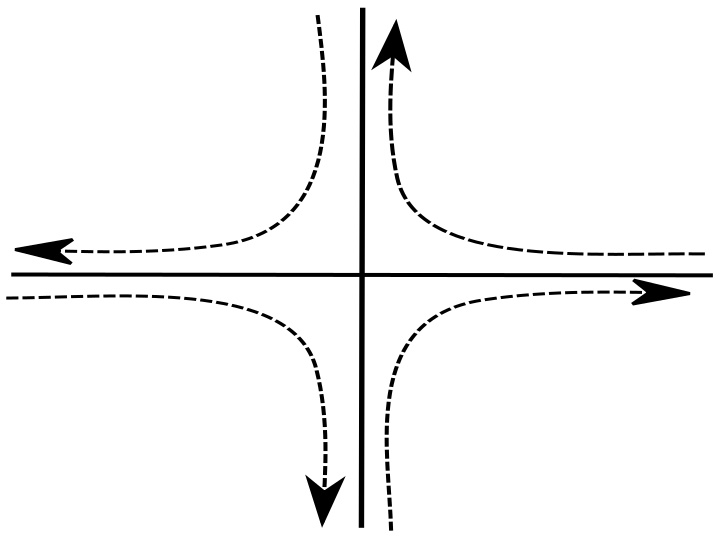}
\put(-45,12){$x$}
\put(-44,2){$y=\gamma(x)$}
\put(-20,2){$\bar{y}=\alpha\gamma(x)$}
\put(-14,12){$C(x):=\gamma\alpha\gamma(x)$}
\hspace{1,5cm}
\includegraphics[scale=0.25]{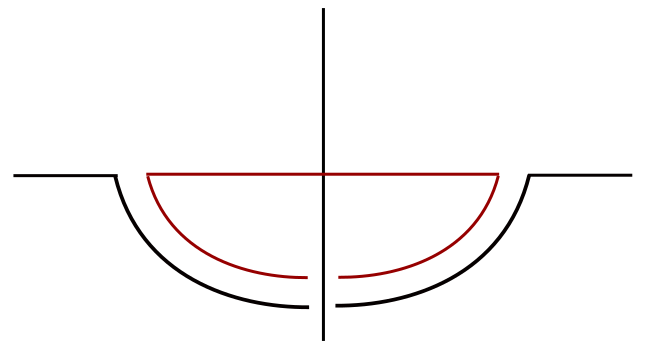}
\caption{A simplification; $x$ and $C(x)$ are consecutive.}
\label{cons}
\end{center}
\end{figure}

\begin{definition}
Assume that $x$ is an edge not lying on a 1-simple component of~$\gm$, and that $x$ and $C(x)$ are intertwined. The operation transforming~$\gm$ into~$\sigma_{x,C(x)}(\gm)$ is called a \textit{simplification}.
\end{definition}

\noindent The name is explained by: 

\begin{lemma}
If $\gm':=\sigma_{x,C(x)}(\gm)$ is a simplification, then ~$S_{\gm'}=S_{\gm}+1$. In other words, a surgery on $\gm$ between $x$ and $C(x)$ creates an additional 1-simple curve in $\gm'$.
\end{lemma}

\begin{proof}
Suppose that $x$ and $C(x)$ are intertwined, a gluing pattern for $\gm$ is given by:
$$W_{\gm}=(tw'_1)\bm{x}(yw'_2\bar{t})\bm{\bar{x}}(w'_3\bar{y})\bm{C(x)}w_4\bm{\overline{C(x)}}.$$
Therefore, by Lemma ~\ref{lemsurg} a gluing pattern for $\gm':=\sigma_{x, C(x)}(\gm)$ is given by: 
$$W_{\gm'}=w'_3\bar{y}\bm{X}yw'_2\bar{t}\bm{\bar{X}}tw'_1\bm{Z}w_4\bm{\bar{Z}}.$$

So, in $\gm'$ we have $C(X)=\gamma\alpha\gamma(X)=\gamma\alpha(y)=\gamma(\bar{y})=X$. It implies that $X$ is the side of simple curve which intersects $\gm'$ only once.
\end{proof}

Now if $\theta_1$ and $\theta_2$ are two 1-simple curves of a one-faced collection, then $\theta_1$ and $\theta_2$ are disjoint; otherwise $\theta_1\cup\theta_2$ would be disjoint from $\gm$, that is absurd since $\gm$ is connected. So the number $S_{\gm}$ of 1-simple curves on a one-faced collection $\gm$ is bounded by the genus $g$ of the underlying surface. Therefore a sequence of simplifications on a one-faced collection stabilizes at a collection on which no simplification can be applied anymore. 

\begin{definition}
A collection $\gm$ is \textit{non simplifiable} if one cannot do a simplification from it, i.e, $x$ and $C(x)$ are always non intertwined.
\end{definition}      
\end{paragraph}

\begin{paragraph}{Order around vertices of a non simplifiable collection:} In this paragraph, we will show that vertices of non simplifiable one-faced collections are of certain types. 

Let $\gm$ be a one-faced collection and $(H,\alpha, \mu, \gamma)$ the permutations associated to $\gm$; $H$ being the set of letters of a gluing pattern for $\gm$.\\
If we fix an origin $x_0\in H$, we then get an order on $H$:
$$x_0<\gamma(x_0)<...<\gamma^{8g-3}(x_0).$$
Therefore, if $v$ is a vertex of $\gm$ defined by a cycle $(t x y z)$ of $\mu$, we get a local order around $v$ by comparing $t$, $x$, $y$ and $z$. Since each letter corresponds to an oriented edge which leaves an angular sector of $v$ (see Figure \ref{avertex}), the local order around $v$ corresponds also to a local order on the four angular sectors around $v$ when running around $\gm$ with $\gamma$. 

\begin{definition} Let $v$ be a vertex defined by the oriented edges $(t,x:=\mu(t),y:=\mu^2(t),z:=\mu^3(t))$ with $t=\min\{t,x,y,z\}$ relatively to an order of edges on $\gm$. Then,
\begin{itemize} 
\item $v$ is a vertex of \textit{Type ~1} if $t<x<y<z$;
\item $v$ is a vertex of \textit{Type ~2} if $t<z<y<x$.
\end{itemize}
Otherwise, the vector $v$ is a vertex of \textit{Type ~3}.
\end{definition}

Up to rotation and change of origin, we have the three cases depicted on Figure ~\ref{type}.
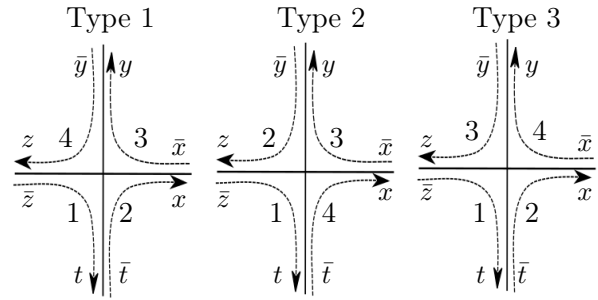
\begin{figure}[htbp]
\begin{center}
\includegraphics[scale=0.3]{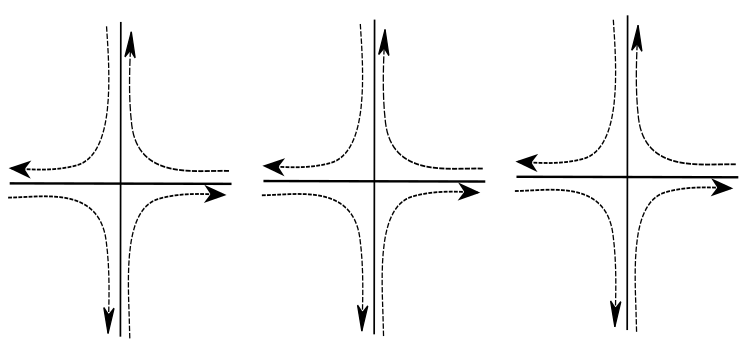}
\put(-71,37){Type 1}
\put(-43,37){Type 2}
\put(-17,37){Type 3}
\put(-71,11){1}
\put(-64,11){2}
\put(-62,21){3}
\put(-72,21){4}
\put(-44,11){1}
\put(-45,21){2}
\put(-36,21){3}
\put(-37,11){4}
\put(-17,11){1}
\put(-10,11){2}
\put(-18,22){3}
\put(-9,22){4}
\put(-77,21){\small{$z$}}
\put(-77,13){\small{$\bar{z}$}}
\put(-70,3){\small{$t$}}
\put(-64,3){\small{$\bar{t}$}}
\put(-57,13){\small{$x$}}
\put(-57,20){\small{$\bar{x}$}}
\put(-64,31){\small{$y$}}
\put(-70,31){\small{$\bar{y}$}}
\put(-51,13){\small{$\bar{z}$}}
\put(-51,21){\small{$z$}}
\put(-43.5,3){\small{$t$}}
\put(-37,3){\small{$\bar{t}$}}
\put(-29.5,14){\small{$x$}}
\put(-29.5,20){\small{$\bar{x}$}}
\put(-37,31){\small{$y$}}
\put(-43.5,31){\small{$\bar{y}$}}
\put(-24,14){\small{$\bar{z}$}}
\put(-24,21){\small{$z$}}
\put(-16.5,3){\small{$t$}}
\put(-11,3){\small{$\bar{t}$}}
\put(-3,14){\small{$x$}}
\put(-3,20.5){\small{$\bar{x}$}}
\put(-10,31){\small{$y$}}
\put(-16.5,31){\small{$\bar{y}$}}
\caption{Different types of vertices of a one-faced collection.}
\label{type}
\end{center}
\end{figure}

\begin{lemma}
A one-faced collection $\gm$ is non simplifiable if and only if all of its vertices are of Type ~1 or Type ~2.
\end{lemma}
\begin{proof}
All we have to do is to write the possible gluing patterns of $\gm$ by figuring out the order of the edges around a vertex and look when consecutive edges are intertwined or not.

\textbf{Case 1:} If $v$ is a vertex of Type ~1, then a gluing pattern for $\gm$ is given by: $$W_{\gm}=w_1\bar{z}tw_2\bar{t}xw_3\bar{x}yw_4\bar{y}z.$$
Therefore, one cheks that $\bar{x}$ and $C(\bar{x})=z$ are not intertwined; so are $\bar{t}$ and $C(\bar{t})=~y$. Hence, no simplification is possible around $v$.

\textbf{Case 2 :} If $v$ is a vertex of Type ~2, the gluing pattern for $\gm$ is $$W_{\gm}=w_1\bar{z}tw_2\bar{y}zw_3\bar{x}yw_4\bar{t}x.$$
Then, $\bar{x}$ and $C(\bar{x})=z$ are not intertwined; so are $\bar{t}$ and $C(\bar{t})=y$. Again, no simplification is possible around $v$ in this case.

\textbf{Case 3:} If $v$ is a vertex of Type ~3, then 
$$W_{\gm}=w_1\bar{z}tw_2\bar{t}xw_3\bar{y}zw_4\bar{x}y.$$
Here, $\bar{t}$ and $C(\bar{t})=y$ are intertwined and a simplification is possible.

So $\gm$ is non simplifiable if and only all is vertices are of Type ~1 or Type ~2. 
\end{proof}
\end{paragraph}
\begin{paragraph}{Number of vertices of Type ~1 and 2 in a non simplifiable one-faced collection:} In \cite{Chap}, G. Chapuy has defined a notion which catches the topology of a unicellular map: \textit{trisection}. We recall one of his results about trisection.

Let $G$ be an unicellular map and $(H,\alpha, \mu, \gamma)$ the permutation associated to $G$. Let $v$ be a degree $d$ vertex of $G$ defined by a cycle $(x_1 x_2...x_d)$ of ~$\mu$, with $x_1=\min\{x_1,..., x_d\}$ relatively to an order on $H$.

If $x_i>x_{i+1}$ we say that we have a \textit{down-step}. Since $x_1=\min\{x_1,..., x_d\}$, one has $x_d>x_1$; the other down-steps around $v$ are called \textit{non trivial}.  
\begin{definition}[G. Chapuy]
A \textit{trisection} is a down-step which is not a trivial one. 
\end{definition}
\begin{lemma}[The trisection lemma; G. Chapuy \cite{Chap}] Let $G$ be unicellular map on a genus $g$ surface. Then $G$ has exactly $2g$ trisections.  
\end{lemma}

Applying the trisection lemma to one-faced collections, we get:
\begin{corollary}
A non simplifiable one-faced collection on $\sg_g$ has $g$ vertices of Type ~2 and $g-1$ vertices of Type ~1.
\end{corollary} 
\begin{proof}
A vertex of Type ~2 (respectively a vertex of Type ~1) has two trisections (respectively zero trisection) (see Figure ~\ref{type}). If $N_i$ is the number of vertices of Type ~ i (i=1,2), by the trisection lemma we have $2N_2=2g$, so $N_2=g$.\\
Since $V_{\gm}=N_1+N_2=2g-1$, it follows that $N_1=g-1$.
\end{proof}
\end{paragraph}
\begin{paragraph}{Repartition of vertices on a non simplifiable collection:} Now, we show that using surgeries, we can re-order the vertices of a non simplifiable one-faced collection.
\begin{definition} 
Let $G$ be a graph. Two vertices are \textit{adjacent} if they share an edge. 
\end{definition}

If $v_1$ and $v_2$ are two vertices represented by the cycles $(abcd)$ and $(efgh)$, respectively, they are adjacent if and only if there exist $x\in\{a,b,c,d\}$ such that $\bar{x}\in\{e,f,g,h\}$.\vspace{0.2cm}
 
We now show that some configurations of vertices "hide" simplifications; that is from those configurations we can create new simplifications after a suitable surgery  without killing the old ones.
\begin{lemma}\label{type2neigh}
Let $\gm$ be a non simplifiable one-faced collection. If $\gm$ contains two vertices of Type ~2 which are adjacent, then there is a sequence of surgeries $\gm=\gm_0\longrightarrow\gm_1\longrightarrow...\longrightarrow\gm_n$ from $\gm$ to ~$\gm_n$ such that $\gm_n$ is non simplifiable and $S_{\gm}<S_{\gm_n}$.
\end{lemma}

\begin{proof} Let $v_1$ and $v_2$ be two adjacent vertices of Type 2 defined by the cycles $(b\bar{f}\bar{g}\bar{a})$ and $(c\bar{d}\bar{e}\bar{b})$, respectively (see Figure ~\ref{hsimple}). 
\begin{figure}[htbp]
\begin{center}
\includegraphics[scale=0.2]{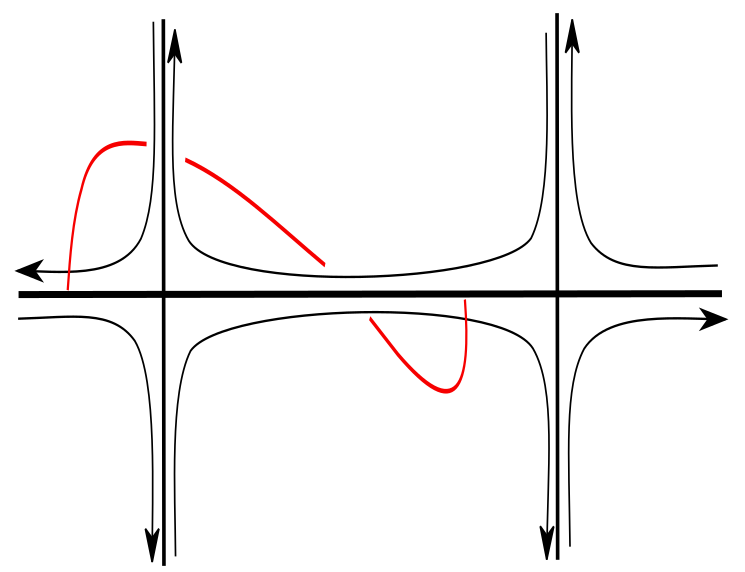}
\put(-40,2){$a$}
\put(-45.7,2){$\bar{a}$}
\put(-27,14){$b$}
\put(-27,24){$\bar{b}$}
\put(-17,2){$c$}
\put(-12,2){$\bar{c}$}
\put(-4,24){$d$}
\put(-4,13){$\bar{d}$}
\put(-17,38){$e$}
\put(-11.5,38){$\bar{e}$}
\put(-45.7,38){$f$}
\put(-39.5,38){$\bar{f}$}
\put(-51,14){$g$}
\put(-51,24){$\bar{g}$}
\includegraphics[scale=0.2]{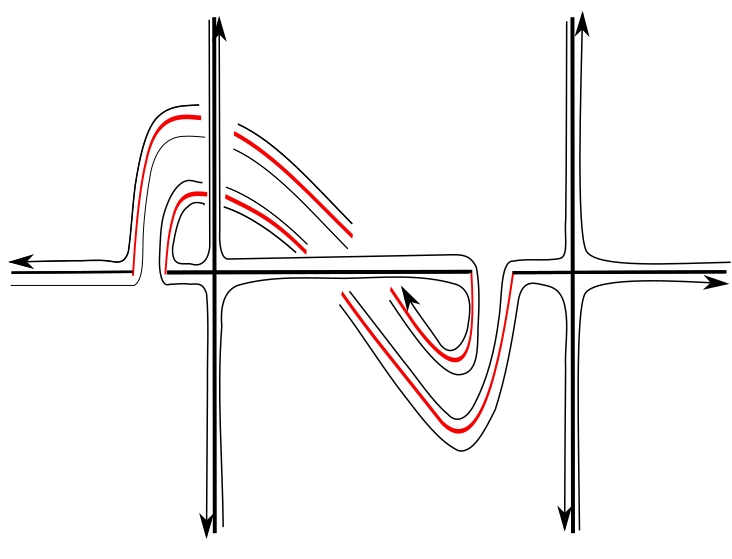}
\put(-35,17){\small{1}}
\put(-35,22){\small{4}}
\put(-39,17){\small{2}}
\put(-39,22){\small{3}}
\put(-14,17){\small{2}}
\put(-14,22){\small{1}}
\put(-9,17){\small{4}}
\put(-9,22){\small{3}}
\caption{Surgery which creates new simplifications. On the left figure, $a<g<f<e<d<\bar{c}$ is the order by which we pass through the eight sectors. On the figure on the right, we focus on the angular order around $v_1$ and $v_2$. The order on the figure on the right comes from that of the figure on the left. At each time we leave the local configuration on the figure on the right, we come back on it in the same way like in the figure on the left.}
\label{hsimple}
\end{center}
\end{figure}
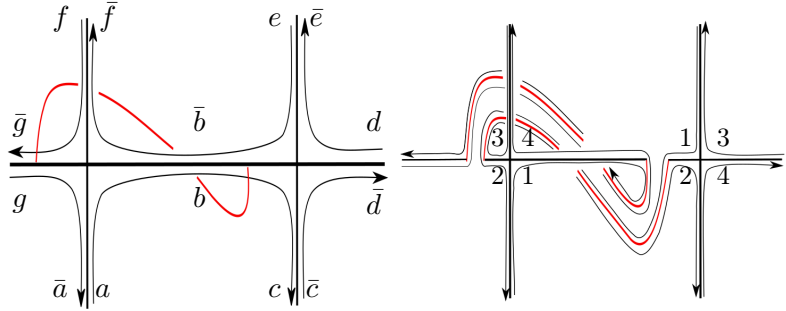

Let us fix an oriented edge as an origin, so that $$a=\min\{a,  \bar{c}, d, e, f, g\};$$ that is the first time we enter in the local configuration is by the oriented edge $a$. Then we have the following order: $$a<b<c<g<\bar{a}<f<\bar{g}<e<\bar{b}<\bar{f}<d<\bar{e}<\bar{c}<\bar{d}.$$ 
Otherwise, it would contradict the fact that the two vertices are of Type ~2. A gluing pattern for $\gm$ is given by:
$$W_{\gm}=w_1a\bm{b}cw_2\bm{g}\bar{a}w_3f\bm{\bar{g}}w_4e\bm{\bar{b}}\bar{f}w_5d\bar{e}w_6\bar{c}\bar{d}.$$

The oriented edges $b$ and $\bar{g}$  are intertwined, so we can define $\gm':=\sigma_{b,\bar{g}}(\gm)$. By Lemma \ref{lemsurg}, a gluing pattern for $\gm'$ is:
$$W_{\gm'}=\bar{a}w_3f\bm{B}cw_2\bm{G}\bar{f}w_5d\bar{e}w_6\bar{c}\bar{d}w_1a\bm{\bar{G}}w_4e\bm{\bar{B}}.$$

The cycles $(\bar{G} \bar{f} B \bar{a})$ and $(\bar{B} c \bar{d} \bar{e})$ define the two vertices of $\gm'$ in Figure ~\ref{hsimple} and the orders around these two vertices are: \[\bar{G}<\bar{a}<B<\bar{f}; \quad \bar{B}<c<\bar{e}<\bar{d}.\]

Therefore, the vertex $(\bar{B}, c, \bar{d}, \bar{e})$ is a vertex of Type ~3 and it implies that $\gm'$ is simplifiable. Indeed the operation on Figure ~\ref{hsimple} does not touch any 1-simple curve of $\gm$ and each simplification increases strictly the number of 1-simple. Let $\gm_n$ be a non simplifiable collection obtained after finitely many simplifications on ~$\gm'$; so $S_\gm<S_{\gm_n}$.    
\end{proof}

Let $v_1$ and $v_2$ be two vertices of Type ~1 and Type ~2 defined by the cycles $(\bar{c} d e f)$ and $(g\bar{a}bc)$, respectively, such that $v_1$ and $v_2$ are adjacent. The local configuration in this case is depicted on Figure ~\ref{figsept} and we assume that $$a=\min\{a, \bar{b}, \bar{d}, \bar{e}, \bar{f}, \bar{g}\}.$$

\begin{figure}[htbp]
\begin{center}
\includegraphics[width=6cm]{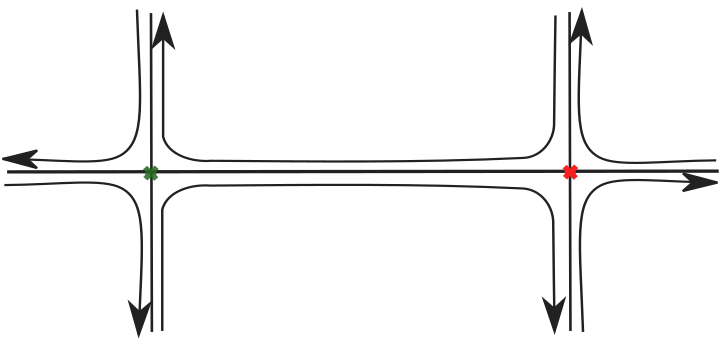}
\put(-2,17){\small{$a$}}
\put(-2,11){\small{$\bar{a}$}}
\put(-11.1,25){\small{$b$}}
\put(-16,25){\small{$\bar{b}$}}
\put(-32,17){\small{$c$}}
\put(-32,11){\small{$\bar{c}$}}
\put(-45,25){\small{$d$}}
\put(-51,25){\small{$\bar{d}$}}
\put(-58,17){\small{$e$}}
\put(-58,11){\small{$\bar{e}$}}
\put(-52,2){\small{$f$}}
\put(-45.7,2){\small{$\bar{f}$}}
\put(-17.3,2){\small{$g$}}
\put(-11.1,2){\small{$\bar{g}$}}
\caption{}
\label{figsept}
\end{center}
\end{figure}
 
\begin{lemma}\label{tictac}
If $\min\{\bar{b}, \bar{d}, \bar{e}, \bar{f}, \bar{g}\}\neq\bar{g}$, then there is a sequence  $\gm=\gm_0\longrightarrow\gm_1\longrightarrow...\longrightarrow\gm_n$ from $\gm$ to $\gm_n$ such that $\gm_n$ is non simplifiable and $S_{\gm}<S_{\gm_n}$.
\end{lemma} 
\begin{proof} 
Since $v_2$ is a vertex of Type ~2, $\min\{\bar{b}, \bar{d}, \bar{e}, \bar{f}, \bar{g}\}$ is different from $\bar{b}$ and $\bar{f}$. 

\textbf{Case 1:} If $\min\{\bar{b}, \bar{d}, \bar{e}, \bar{f},  \bar{g}\}=\bar{d}$, the fact that the vertices are of Type ~1 and 2 implies that the local order is either $$a<b<\bar{d}<e<\bar{e}<f<\bar{g}<\bar{a}<\bar{f}<\bar{c}<g<\bar{b}<c<d,$$ or $$a<b<\bar{d}<e<\bar{g}<\bar{a}<\bar{e}<f<\bar{f}<\bar{c}<g<\bar{b}<c<d.$$

\hspace{0.05cm}\textbf{Sub-case 1:} If $a<b<\bar{d}<e<\bar{e}<f<\bar{g}<\bar{a}<\bar{f}<\bar{c}<g<\bar{b}<c<d$, then $$W_{\gm}=w_1\bm{a}bw_2\bar{d}\bm{e}w_3\bm{\bar{e}}fw_4\bar{g}\bm{\bar{a}}w_5\bar{f}\bar{c}gw_6\bar{b}cd$$ is a gluing pattern for $\gm$.

The oriented edges $a$ and $\bar{e}$ are intertwined and a gluing pattern for $\gm':=\sigma_{a,\bar{e}}(\gm)$ is given by \[W_{\gm'}=w_3\bm{A}bw_2\bar{d}\bm{E}w_5\bar{f}\bar{c}gw_6\bar{b}cdw_1\bm{\bar{E}}fw_4\bar{g}\bm{\bar{A}}.\] 
The cycles $(bcg\bar{A})$ and $(Ef\bar{c}d)$ define the two vertices in Figure ~$\ref{figneuf}$. Moreover, $b<g<c<\bar{A}$ and $E<\bar{c}<d<f$ that is they are vertices of Type ~3.\\
\begin{figure}[htbp]
\begin{center}
\includegraphics[width=6cm]{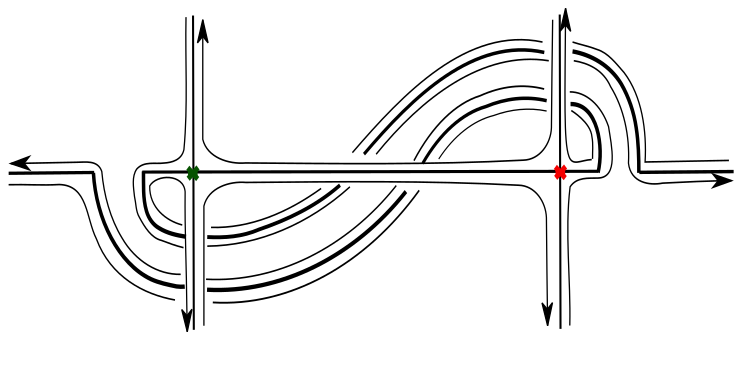}
\put(-47,18.5){\small{1}}
\put(-42.5,13){\small{2}}
\put(-42.5,18.5){\small{3}}
\put(-47,13.5){\small{4}}
\put(-13.5,18.5){\small{1}}
\put(-17.5,12.5){\small{2}}
\put(-17.5,18.5){\small{3}}
\put(-13,13.5){\small{4}}
\caption{}
\label{figneuf}
\end{center}
\end{figure}
\noindent Therefore, $\gm'$ is simplifiable and there is sequence of simplification from $\gm'$ to $\gm_n$ such that  $\gm_n$ is non simplifiable and $N_{\gm_n}>N_{\gm'}=N_{\gm}$. The equality $N_{\gm'}=N_{\gm}$ holds since the surgery in this case does not touch a 1-simple curve.

\hspace{0.05cm}\textbf{Sub-case 2:} If $a<b<\bar{d}<e<\bar{g}<\bar{a}<\bar{e}<f<\bar{f}<\bar{c}<g<\bar{b}<c<d$, then a gluing pattern for $\gm$ is 
$$W_{\gm}=w_1\bm{a}bw_2\bar{d}ew_3\bar{g}\bm{\bar{a}}w_4\bar{e}\bm{f}w_5\bm{\bar{f}}\bar{c}gw_6\bar{b}cd.$$
Here again, the oriented edges $a$ and $f$ are intertwined and a gluing pattern for $\gm':=\sigma_{a,f}(\gm)$ is given by: 
$$W_{\gm'}=w_4\bar{e}\bm{A}bw_2\bar{d}ew_3\bar{g}\bm{\bar{A}}\bar{c}gw_6\bar{b}cdw_1\bm{F}w_5\bm{\bar{F}}.$$
The two vertices in Figure \ref{figdix} are defined by the cycles $(bcg\bar{A})$ and $(A\bar{c}de)$. Moreover, $b<\bar{A}<g<c$ and $A<e<\bar{c}<d$. It follows that the vertex $(A\bar{c}de)$ is a vertex of Type ~3 and therefore, $\gm'$ is simplifiable.
 \begin{figure}[htbp]
\begin{center}
\includegraphics[width=6cm]{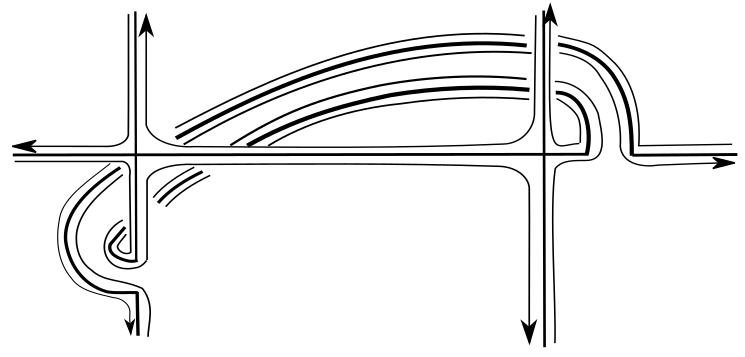}
\put(-15.5,18){\small{1}}
\put(-15,12){\small{2}}
\put(-20,12){\small{3}}
\put(-20,18){\small{4}}
\put(-52,17.5){\small{1}}
\put(-47,13){\small{2}}
\put(-47.5,17.5){\small{3}}
\put(-52,13){\small{4}}
\caption{}
\label{figdix}
\end{center}
\end{figure}

\textbf{Case 2:} if $\min\{\bar{b}, \bar{d}, \bar{e}, \bar{f}, \bar{g}\}=\bar{e}$, then the local order is given by: $$a<b<\bar{e}<f<\bar{g}<\bar{a}<\bar{f}<\bar{c}<\bar{g}<\bar{b}<c<d<\bar{d}<e,$$ 
and a gluing pattern for $\gm$ is given by:
$$W_4=w_1\bm{a}bw_2\bar{e}fw_3\bar{g}\bm{\bar{a}}w_4\bar{f}\bar{c}gw_5\bar{b}c\bm{d}w_6\bm{\bar{d}}e.$$

In this case, $a$ and $d$ are intertwined. A gluing pattern for $\gm':=\sigma_{a,d}(\gm)$ is given by:
$$W_{\gm'}=w_4\bar{f}\bar{c}gw_5\bar{b}c\bm{A}bw_2\bar{e}fw_3\bar{g}\bm{\bar{A}}ew_1\bm{D}w_6\bm{\bar{D}}.$$
The cycles $(\bar{c}Aef)$ and $(g\bar{A}bc)$ represent the two vertices in Figure ~\ref{figonze} and $\bar{c}<A<f<e$. So, the vertex $(\bar{c}Aef)$ is a vertex of Type ~3 and $\gm'$ is simplifiable. 
\begin{figure}[htbp]
\begin{center}
\includegraphics[width=6cm]{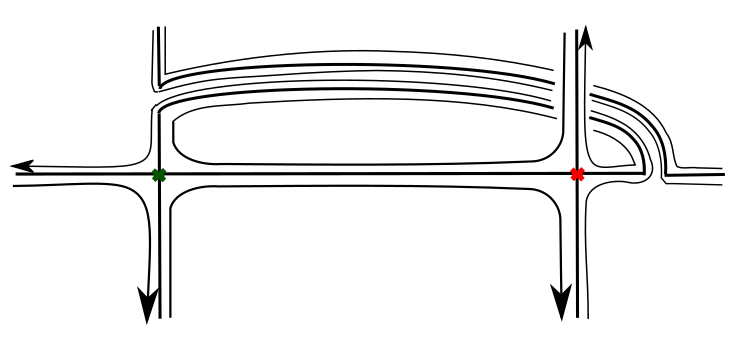}
\put(-11.5,14.7){\small{1}}
\put(-11,10){\small{2}}
\put(-17,9.5){\small{3}}
\put(-17,15.3){\small{4}}
\put(-50,9.5){\small{1}}
\put(-50,15.5){\small{2}}
\put(-46,9.5){\small{3}}
\put(-45,15.5){\small{4}}
\caption{}
\label{figonze}
\end{center}
\end{figure}
\end{proof}

\begin{definition}(see Figure \ref{figsept}) Let $v_1$ and $v_2$ be two adjacent vertices of Type ~1 and 2, respectively. We say that we have a \textit{good order} around $v_1$ and $v_2$ if $\bar{g}=\min\{\bar{b}, \bar{d}, \bar{e}, \bar{f}, \bar{g}\}$. 
\end{definition}
\begin{definition}
A one-faced collection $\gm$ is \textit{almost toral} if:
\begin{itemize}
\item $\gm$ is non simplifiable, 
\item no two vertices of Type ~2 are adjacent,
 \item the local orders around two adjacent vertices of Type ~1 and 2 are good.
 \end{itemize}
\end{definition}
\begin{lemma}\label{almideal}
Let $\gm$ be a non simplifiable one-faced collection. Then there is a sequence of surgeries $\gm_0=\gm\longrightarrow\gm_1...\longrightarrow\gm_n$ such that $\gm_n$ is an almost toral  one-faced collection.
\end{lemma}
\begin{proof}
If $\gm$ is not almost toral, by Lemma ~\ref{type2neigh} and Lemma ~\ref{tictac} there is a sequence of surgeries $\gm_0=\gm\longrightarrow\gm_1\longrightarrow...\longrightarrow \gm_n$ such that $\gm_n$ is non simplifiable and $S_{\gm}<S_{\gm_n}$, i.e, we create new 1-simple curves after some suitable surgeries. Since the number of 1-simple curves is bounded by the genus, those operations stabilize to an almost toral one-faced collection.
\end{proof}
Now, we are going to improve the configuration of the vertices of an almost toral one-faced collection. 

\begin{lemma}\label{idealstep}
Let $\gm$ be an almost toral one-faced collection, $v_1$ and $v_2$ be two adjacent vertices of Type ~1 and Type ~2, respectively; with $a=\min\{a, \bar{b}, \bar{d}, \bar{e}, \bar{f}, \bar{g}\}$ (see Figure ~\ref{figsept}). 

If $x:=\min\{\bar{b}, \bar{d}, \bar{e}, \bar{f}\}$ is adjacent to a vertex of Type ~2, then there is a sequence of surgeries $\gm_0=\gm\longrightarrow\gm_1\longrightarrow...\longrightarrow\gm_n$ such that $\gm_n$ is almost toral and $N_{\gm}<N_{\gm_n}$.
\end{lemma}
\begin{proof}
If $x:=\min\{\bar{b}, \bar{d}, \bar{e}, \bar{f}\}$ is adjacent to a vertex $v:=(x\bar{y}\bar{t}\bar{u})$, a gluing pattern of $\gm$ is given by (see Figure ~\ref{serp}): 
\[W_{\gm}=w_1\bm{a}w_2\bm{\bar{a}}w_3u\bm{x}w_4\bm{\bar{x}}\bar{y}.\]

\begin{figure}[htbp]
\begin{center}
\includegraphics[width=5cm]{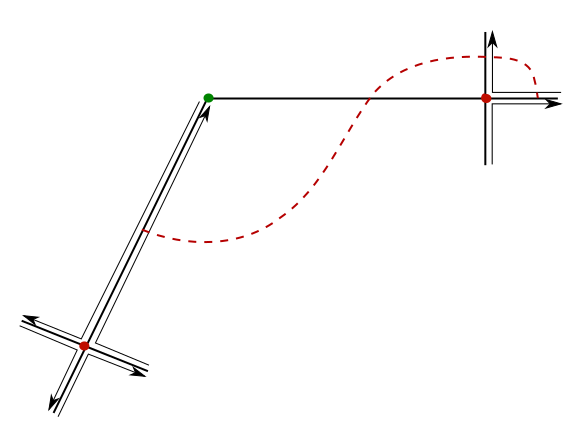}
\put(-3,30){\small{$a$}}
\put(-3,25.5){\small{$\bar{a}$}}
\put(-38,14){\small{$x$}}
\put(-42,15){\small{$\bar{x}$}}
\put(-37,6){\small{$u$}}
\put(-40,2.6){\small{$\bar{u}$}}
\put(-44.5,-0.5){\small{$t$}}
\put(-47.7,0.5){\small{$\bar{t}$}}
\put(-50,7.5){\small{$y$}}
\put(-49,11){\small{$\bar{y}$}}
\caption{}
\label{serp}
\end{center}
\end{figure}
Since $v$ is a vertex of Type ~2, $x<\bar{u}<\bar{t}<\bar{y}$. It implies that $\bar{t}\in w_4$ and $\bar{u}\in w_4$.

The oriented edges $a$ and $x$ are intertwined and \[W_{\gm'}=w_3u\bm{A}w_2\bm{\bar{A}}\bar{y}w_1\bm{X}w_4\bm{\bar{X}}\]
is a gluing pattern for $\gm':=\sigma_{a,x}(\gm)$. The vertex $v$ in $\gm'$ is defined by the cycle $(A\bar{y}\bar{t}\bar{u})$ and one checks that $A<\bar{y}<\bar{u}<\bar{t}$; that is $v$ is a vertex of Type ~3 and $\gm'$ is simplifiable. Hence, there is a sequence of simplification which strictly increases the number of 1-simple curves. 
\end{proof}

\begin{definition}
Let $\gm$ be a one-faced collection. We say that $\gm$ is a \textit{toral} one-faced collection if ~$\gm$ is an almost toral  one-faced collection and if every vertex of Type ~1 is adjacent to at most two vertices of Type ~2. 
\end{definition}
\begin{lemma}\label{tictactoe}
Let $\gm$ be an almost toral  one-faced collection. Then there is a sequence of surgeries $\gm_0=\gm\longrightarrow\gm_1\longrightarrow...\longrightarrow\gm_n$ such that $\gm_n$ is a toral  one-faced collection.   
\end{lemma} 

\begin{proof}
Let $v$ be a vertex of $\gm$ of Type ~1. If $v$ is adjacent to $4$ vertices of Type ~2 or $3$ vertices all of which are of Type ~2, Lemma ~\ref{idealstep} implies that there is a sequence of surgeries $\gm_0=\gm\longrightarrow\gm_1\longrightarrow...\longrightarrow\gm_n$ such that $S_{\gm}<S_{\gm_n}$. Since the number of 1-simple curves is bounded by the genus $g$, there is a sequence of surgeries $\gm_0=\gm\longrightarrow\gm_1\longrightarrow...\longrightarrow\gm_n$ such that $\gm_n$ is almost toral and such that every vertex of Type ~1 adjacent to three vertices of Type ~2 is also adjacent to a fourth of Type ~1.  The local configuration around those vertices is depicted in Figure ~\ref{tabletrois}, with $\bar{e}=\min\{\bar{d}, \bar{e}, \bar{f}\}$ . A gluing pattern for $\gm_n$ is given by:
\[W_{\gm_n}=w_1\bm{a}w_2\bm{\bar{a}}w_3\bm{d}w_4\bm{\bar{d}}.\]

The oriented edges $a$ and $d$ are intertwined and the surgery $\sigma_{a,d}(\gm_n)$ decreases the number of adjacent vertices to $v$ (see Figure ~\ref{tabletrois}). Following this process, we get a toral one-faced collection after finitely many surgeries. 

\begin{figure}[htbp]
\begin{center}
\includegraphics[width=5.5cm]{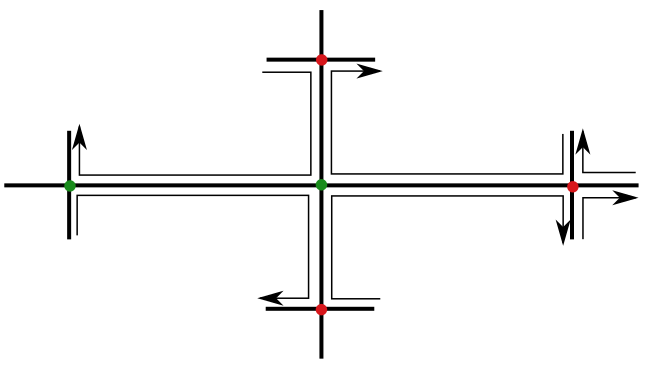}\hspace{0.5cm}
\includegraphics[width=6cm]{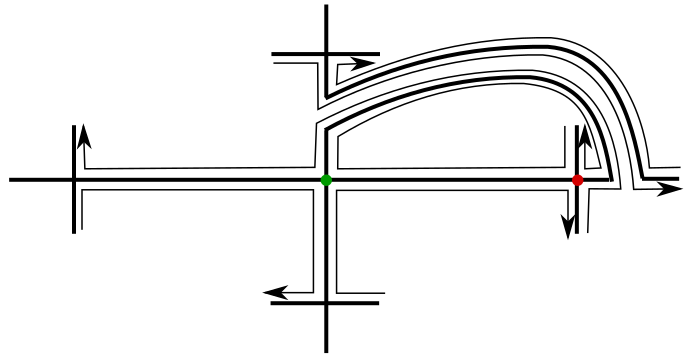}
\put(-68,17){\tiny{$a$}}
\put(-68,12.5){\tiny{$\bar{a}$}}
\put(-71,20){\tiny{$b$}}
\put(-75,20){\tiny{$\bar{b}$}}
\put(-84,17){\tiny{$c$}}
\put(-84,12.5){\tiny{$\bar{c}$}}
\put(-93,20){\tiny{$d$}}
\put(-97,20){\tiny{$\bar{d}$}}
\put(-106,17){\tiny{$e$}}
\put(-106,12.5){\tiny{$\bar{e}$}}
\put(-97,9){\tiny{$f$}}
\put(-93,9){\tiny{$\bar{f}$}}
\put(-71.5,10){\tiny{$g$}}
\put(-75.5,10){\tiny{$\bar{g}$}}
\put(-34,17.5){\tiny{$1$}}
\put(-34.5,13){\tiny{$2$}}
\put(-30.5,13){\tiny{$3$}}
\put(-30.5,17.5){\tiny{$4$}}
\put(-9.3,17.5){\tiny{$1$}}
\put(-8.8,13){\tiny{$2$}}
\put(-12.5,13){\tiny{$3$}}
\put(-12.5,17.5){\tiny{$4$}}
\caption{}  
\label{tabletrois}
\end{center}
\end{figure}
\end{proof}
\begin{lemma}\label{ideal}
Let $\gm$ be a toral  one-faced collection in $\sg_{g+1}$. Then $\gm=(\gm',x)\#\gm_{\mathbb{T}}$ where $\gm'$ is a one-faced collection in $\sg_{g}$. 
\end{lemma}
\begin{proof} We have to show that there is a vertex of Type ~2 which is adjacent to exactly one vertex of Type ~1.

Assume that every vertex of Type ~2 is adjacent to at least two vertices of Type ~1. Let $N_1$ and $N_2$ be the number of vertices of Type ~1 and Type ~2 respectively, and let $N_{1,2}$ be the number of pairs of vertices of Type ~1 and Type ~2 which are adjacent.
 
Since $\gm$ is a toral  one-faced collection, any vertex of Type ~1 has at most two vertices of Type ~2. It implies that,  \[N_{1,2}<2N_1.\]

On the other part, we have assumed that every vertex of Type ~2 is adjacent to at least two vertices of Type ~1. Therefore, \[2N_2 \leq N_{1,2}.\]

Combining the two inequalities above, we get $N_2\leq N_1$ which contradicts the fact that we have ~$g-1$ vertices of Type ~1 and $g$ vertices of Type ~2 in a non simplifiable one-faced collection.

So, there is a vertex $v_0$ of Type ~2 which is adjacent to exactly one vertex $v_1$ of Type ~1. As vertices of Type ~2 are not adjacent, $v_0$ lies on a 1-simple curve. $\bm{(*)}$ \vspace{0.2cm}

Next, we show that $v_1$ can be transformed into a self-intersection point (if it is not the case) by a surgery. Assume that $v_1$ is not a self-intersection point. Then  a gluing pattern for $\gm$ is given by:
$$W_{\gm}=w_1x\bm{y}w_2\bm{\bar{y}}\bm{z}w_3\bm{\bar{z}}tw_4\bar{t}\bar{x};$$
where $v_1$ is defined by the cycle $(yzt\bar{x})$ (Figure ~\ref{selfint}). 
\begin{figure}[htbp]
\begin{center}
\includegraphics[scale=0.12]{consecutifdex.png} \put(-30,14){\small{$\bar{x}$}} \put(-30,8){\small{$x$}} \put(-20,1){\small{$y$}} \put(-13.5,1){\small{$\bar{y}$}} \put(-2,8){\small{$z$}} \put(-2,13.5){\small{$\bar{z}$}} \put(-13,20){\small{$t$}} \put(-19.5,20){\small{$\bar{t}$}}\hspace{1.5cm}
\includegraphics[scale=0.15]{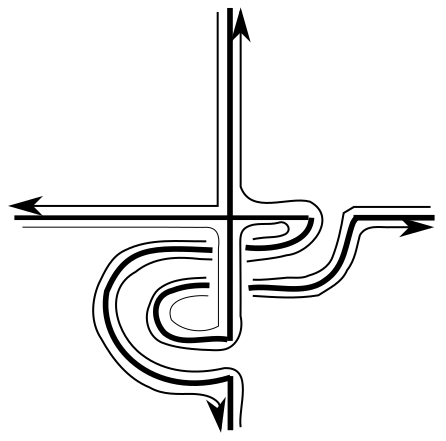}
\caption{Transforming a Type ~1 vertex to a self-intersection point.}
\label{selfint}
\end{center}
\end{figure}

The oriented edges $y$ and $z$ are intertwined and a gluing pattern for $\gm':=\sigma_{y,z}(\gm)$ is given by: $$W_{\gm'}=\bm{Y}w_2\bm{\bar{Y}}tw_4\bar{t}\bar{x}w_1x\bm{Z}w_3\bm{\bar{Z}}$$
The vertex $v_1$ in $\gm'$ is defined by the cycle $(Yt\bar{x}Z)$ and $Y<t<\bar{x}<Z$; that is $v_1$ is still a vertex of Type ~1. Moreover, $v_1$ get transformed to a self-intersection point. So, the surgery on $\gm$ between $y$ and $z$ has transformed $v_1$ to a self-intersection point of Type ~1. $\bm{(**)}$

Finally, $\bm{(*)}$ and $\bm{(**)}$ implies that $$\gm=(\gm',x)\#(\gm_{\mathbb{T}},x_0);$$ with $\gm'$ a one-faced collection on $\sg_{g-1}$. 
\end{proof}
\end{paragraph}
\end{section}
\begin{section}{Proof of the main theorems}\label{sect5}
In this section, we prove Theorem ~\ref{thm1}, Theorem ~\ref{thm2} and Theorem ~\ref{diam}. We recall that the graph $K_g$ is the graph whose vertices are homeomorphism classes of one-faced collections on $\sg_g$, and on which two vertices $\gm_1$ and $\gm_2$ are connected by an edge if there is a surgery which transforms $\gm_1$ into $\gm_2$ (if a surgery on $\gm$ fix $\gm$, we do not put a loop).  The graph $\widehat{K}_g:=\displaystyle{\sqcup_{i\leq g}K_i}$ is the disjoint union of the graphs ~$K_i$ on which we add an edge between two one-faced collections ~$\gm_1$ and ~$\gm_2$ on $\sg_i$ and $\sg_{i+1}$ respectively if $\gm_2$ is a connected sum of $\gm_1$ with the one-faced on the torus.

The following proposition is the main technical result: it easily implies Theorem ~\ref{thm2}. It also implies Theorem ~\ref{thm1} with a bit of extra-work and its proof uses most lemmas of Section \ref{sect4}. 
\begin{prop}\label{prop} Let $\gm$ be a one-faced collection on $\sg_{g+1}$. Then there is a finite sequence of surgeries $\gm:=\gm_0\longrightarrow...\longrightarrow\gm_n$ and a one-faced collection $\gm'$ on $\sg_g$ with a marked edge $x$ such that $\gm_n=(\gm',x)\#\gm_{\mathbb{T}}$.
\end{prop}
\begin{proof}
Let $\gm$ be a one-faced collection on $\sg_g$, there is a sequence of surgeries $\gm\longrightarrow...\longrightarrow\gm_1$ such that ~$\gm_1$ is non simplifiable. By Lemma ~\ref{almideal}, there is a sequence of surgeries $\gm_1\longrightarrow...\longrightarrow\gm_2$ such that $\gm_2$ is almost toral  and by Lemma ~\ref{idealstep}, there is a sequence of surgeries $\gm_2\longrightarrow...\longrightarrow\gm_3$ such that $\gm_3$ is toral. Lemma \ref{ideal} implies that $\gm_3=(\gm',x)\#\gm_{\mathbb{T}}$ with $\gm'$ a one-faced collection in $\sg_{g-1}$.
\end{proof}
We can now prove Theorem ~\ref{thm2}, which states that for every $g$ the graph ~$\widehat{K}_g$ is connected. 

\begin{proof}[Proof of Theorem ~\ref{thm2}] Let $\gm \in K_g$. By Proposition ~\ref{prop}, there is path in $K_g$ from $\gm$ to ~$(\gm',x)\#\gm_{\mathbb{T}}$ where $\gm' \in K_{g-1}$. Thus, there is path in $\widehat{K}_g$ from ~$\gm$ to $\gm'$. By induction on ~$g$, we deduce a path from $\gm$ to $\gm_{\mathbb{T}}$. So,  $\widehat{K}_g$ is connected.    
\end{proof}

Now, we turn to the proof of Theorem ~\ref{thm1}. Let us start with some preliminaries. 
\begin{lemma}\label{side}
Let $\gm$ be a one-faced collection on $\sg_g$ and $x$ an oriented edge of $\gm$. Then there is a surgery from $(\gm,x)\#\gm_{\mathbb{T}}$ to $(\gm,\bar{x})\#\gm_{\mathbb{T}}$. 
\end{lemma}
Lemma ~\ref{side} states that up to surgery the connected sum of $\gm$ with $\gm_{\mathbb{T}}$ depends only on the edge we choose on $\gm$ but not on its orientation. 

\begin{proof}
Let $W_{\gm}:=w_1xw_2\bar{x}$ be a gluing pattern for $\gm$, $\gm_1:=(\gm,x)\#\gm_{\mathbb{T}}$ and $\gm_2:=(\gm,\bar{x})\#\gm_{\mathbb{T}}$. We recall that $W_{\gm_{\mathbb{T}}}=ab\bar{a}\bar{b}$ is a gluing pattern for $\gm_{\mathbb{T}}$. By Lemma \ref{sum},
$$W_{\gm_1}=x_1w_1\bar{x}_1x_2w_2\bar{x_2}a_1b\bar{a}_1a_2\bar{b}\bar{a}_2$$
and
$$W_{\gm_2}=x_1w_2\bar{x}_1x_2w_1\bar{x_2}a_1b\bar{a}_1a_2\bar{b}\bar{a}_2$$ are gluing patterns of $\gm_1$ and $\gm_2$, respectively.

Thus, in $W_{\gm_1}$, $x_1$ and $x_2$ are intertwined, $\gm':=\sigma_{x_1,x_2}(\gm_1)$ is one-faced, with gluing pattern
$$W_{\gm'}=x_1w_2\bar{x}_1x_2w_1\bar{x}_2a_1b\bar{a}_1a_2\bar{b}\bar{a}_2.$$ We check that $W_{\gm'}=W_{\gm_1}$. So, $\sigma_{x_1,x_2}(\gm_1)=\gm_2$. 
\end{proof}

\begin{figure}[htbp]
\begin{center}
\includegraphics[scale=0.17]{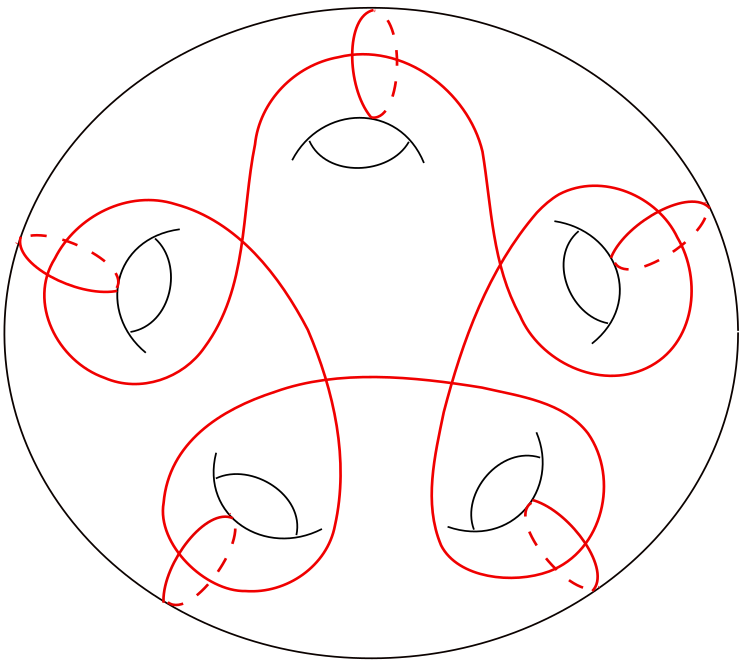}
\caption{The 5-necklace $N_5$.}
\label{neck}
\end{center}
\end{figure}
\begin{definition}
We call \textit{g-necklace} the homeomorphism class of the one-faced collection on $\sg_g$, denoted by $N_g$, with $g$ 1-simple curves and one spiraling curve ~$\eta$ with $g-1$ self intersection points (see Figure \ref{neck} for the 5-necklace). 
\end{definition}
\begin{remark}
Intersection points between 1-simple curves and $\gamma$ in $N_g$ are of Type ~2; the others are Type ~1 vertices. There are $g-1$ vertices of Type ~2 which are adjacent to exactly one vertex of Type ~1 and one \textit{special vertex} of Type ~2 which is adjacent to two vertices of Type ~1.
\end{remark}
We can now prove Theorem ~\ref{thm1}, namely given $\gm_1$ and $\gm_2$ are two one-faced collections on a genus $g$ surface $\sg_g$ there is a finite sequence of surgeries from $\gm_1$ to $\gm_2$. 

\begin{proof}[Proof of Theorem ~\ref{thm1}] We give a proof by induction on $g$. Assume $K_g$ is connected.
Let $\gm$ an one-faced collection on $\sg_{g+1}$. By Proposition ~\ref{prop}, there exists a sequence of surgeries $\gm=\gm_0\longrightarrow...\longrightarrow \gm_n$ where $\gm_n$ is of the form $(\gm',x)\#\gm_{\mathbb{T}}$.

Since we have assumed that $K_g$ is connected, then there is a sequence of surgeries $\gm'=\gm'_0\longrightarrow...\longrightarrow \gm'_n=N_g$ from $\gm'$ to $N_g$ (the $g$-necklace). This sequence lifts to a sequence $\gm=(\gm',x)\#\gm_{\mathbb{T}}\longrightarrow...\longrightarrow (N_g,x_n)\#\gm_{\mathbb{T}}$ of surgeries on $\sg_{g+1}$. Indeed if $x$ is an oriented edge of $\gm$, then $x$ brokes into two oriented edges $x_1$ and $x_2$. The surgery $\sigma_{x,y}(\gm')$ (respectively $\sigma_{z,y}(\gm')$) lift to $\sigma_{x_1,y}(\gm)$ (respectively $\sigma_{z,y}(\gm)$).     

By Lemma ~\ref{side}, up to surgery the way we glue $\gm_{\mathbb{T}}$ on $N_g$ depends only on the edges of $N_g$ but not on there sides. It follows that there are three situations depending whether:
\begin{itemize}
\item $x_n$ is the side of an edge connecting two vertices of Type ~1, or one vertex of Type ~1 and the special vertex of Type ~2,
\item $x_n$ is the side of an edge connecting one vertex of Type ~1 and one vertex of Type ~2 which is not the special one.
\item $x_n$ lies on a 1-simple closed curve.
\end{itemize}
The first situation leads to the $(g+1)$-necklace $N_{g+1}$, and for the other two situations there is a path to the $(g+1)$-necklace. We give the paths for the genus $5$ case in Figure ~\ref{suite}; the other cases inductively follow the same sequence of surgeries.

Since $K_1$ (a single vertex) is connected, by induction $K_g$ is also connected. 
\begin{figure}[htbp]
\begin{center}
\[
   \xymatrix{
       \includegraphics[scale=0.14]{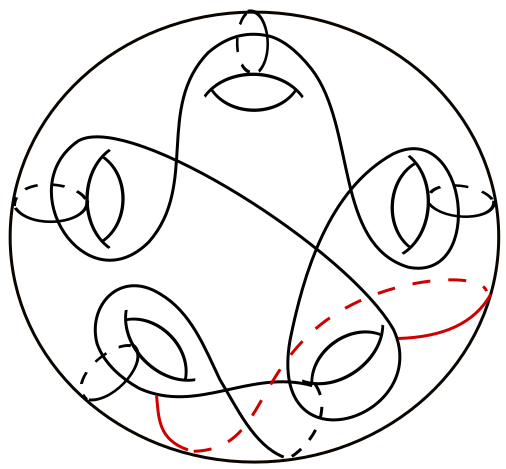}
   & \includegraphics[scale=0.14]{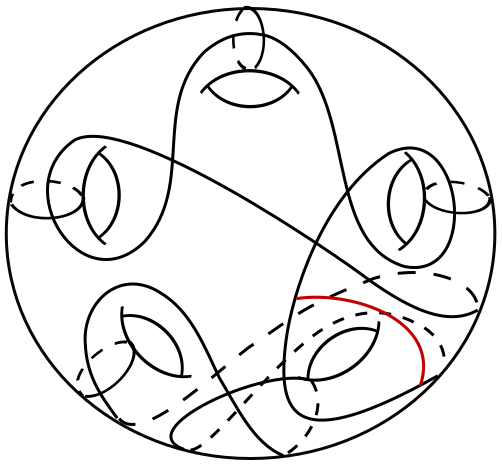}
  & \includegraphics[scale=0.14]{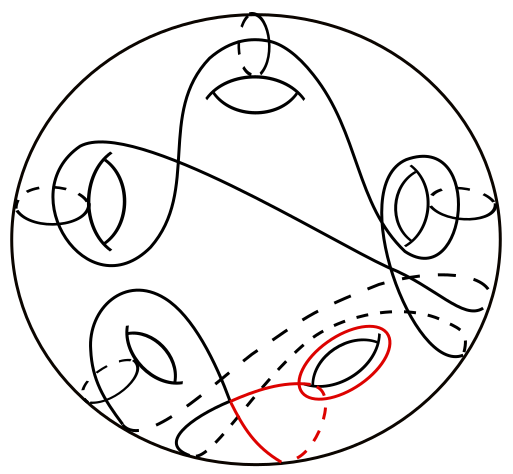} \\
        \includegraphics[scale=0.14]{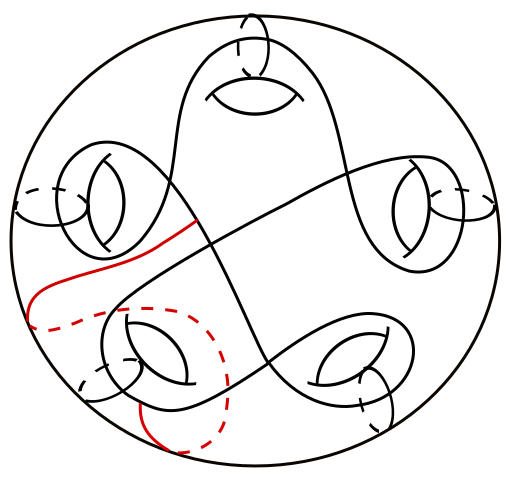}
  & \includegraphics[scale=0.14]{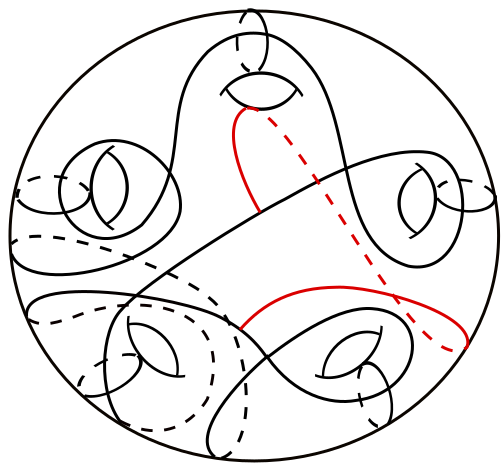}
  & \includegraphics[scale=0.14]{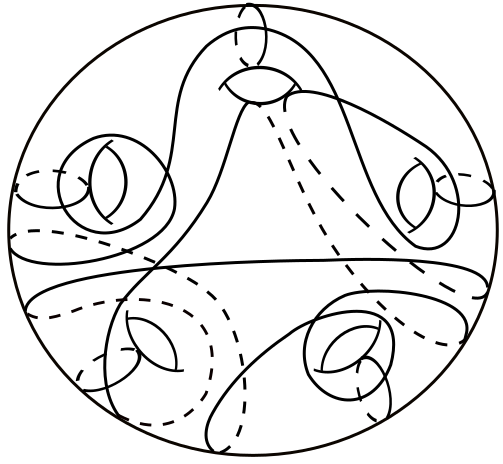}
            }
\put(-72,8){\huge{$\rightarrow$}}
\put(-36,8){\huge{$\rightarrow$}}
\put(-72,-25){\huge{$\rightarrow$}}
\put(-36,-25){\huge{$\rightarrow$}}
\]
\caption{Sequence of surgeries to the necklace. The sequence on the first arrow goes from the case where $\gm_{\mathbb{T}}$ is glued on a 1-simple curve to the case where $\gm_{\mathbb{T}}$ is glued between one vertex of Type ~1 and one vertex of Type ~2. The second arrow leads to the $(g+1)$-necklace. The red arcs are the arcs on which we apply surgeries.}
\label{suite}
\end{center}
\end{figure}  
 \end{proof}

Now we turn to the question of the diameter of $K_g$ that we denote by $D_g$. We prove 
\begin{thm}
For every $g$ we have $D_g\leq 3g^2+9g-12$.
\end{thm}
\begin{proof}
Let $d_{g}:=\max\{d(\gm,N_g)\}$ be the maximal distance to the necklace. By Proposition ~\ref{prop}, if $\gm$ is a one-faced collection, there is sequence $s_n$ of surgeries from $\gm$ to $\gm_n$ such that $\gm_n$ is toral. In this sequence, we have three kind of steps:
\begin{itemize}
\item making an apparent simplification on a vertex of Type ~3; let $m$ be their numbers,
\item  making a hidden simplification, that is a simplification which follows a suitable surgery as in Lemma \ref{type2neigh}; let $n$ be their numbers,
\item making a surgery which are not followed by simplification as in Figure ~\ref{tabletrois}; let $k$ be their numbers.    
\end{itemize}
It follows that the length $l(s_n)$ is equal to $m+2n+k$; with $m+n\leq g$ and $k\leq g-1$ (since the last step correspond to a surgery around vertices of Type ~1). 

The maximum is reached when every simplification follows a suitable surgery; that is $m=0$ and $n=g$. So we have $l(s_n)\leq 3g-1$. 

Since $\gm_n=(\gm',x)\#\gm_{\mathbb{T}}$, it follows that $\gm$ is at most  at distance $3g-1+d_{g-1}$ of $(N_{g-1},y)\#\gm_{\mathbb{T}}$.  So, $$d(\gm,N_g)\leq 3g+3+d_{g-1},$$ for $(N_{g-1},y)\#\gm_{\mathbb{T}}$ is at most  at distance $4$ of $N_g$ (see Figure \ref{suite}). Hence $$d_g\leq3g+3+d_{g-1};$$ and by induction on $g$ \[D_g\leq 2d_g\leq 3g^2+9g-12.\]
\end{proof}
\begin{paragraph}{Non hyperbolicity of $K_{\infty}$} Let $(X,d)$ a totally geodesic metric space, that is every pair of points in $X$ are joint by a geodesic. 

\begin{definition} A geodesic triangle is a triple $(\eta_1,\eta_2,\eta_3)$ of geodesics\\ $\eta_i:[0,1]\longrightarrow X$ such that:
\[\eta_1(1)=\eta_2(0);\quad \eta_2(1)=\eta_3(0);\quad \eta_3(1)=\eta_1(0).\]

Let $\delta\in\R_{+}$ and $T:=(\eta_1,\eta_2,\eta_3)$ geodesic triangle of $X$. We say that $T$ is $\delta$-\emph{thin} if the $\delta$-neighborhood of the union of two geodesics of $T$ contain the third.

A metric space $(X,d)$ is \emph{Gromov hyperbolic} if there exist $\delta\geq0$ such that every geodesic triangle $T$ is $\delta$-thin.    
\end{definition}

\noindent For more details on Gromov hyperbolic spaces, see \cite{Ghys}.

We show that $K_{\infty}$ is not Gromov hyperbolic by giving triangles on $K_{\infty}$ which are not $\delta$-thin for a given $\delta$. We denote by $d$ the distance on $K_{\infty}$. We recall that for a unicellular collection $\gm$, $S_{\gm}$ denotes the number of $1$-simple curves of $\gm$.

\begin{lemma} Let $\gm_1$ and $\gm_2$ two unicellular collections. Then, \[d(\gm_1,\gm_2)\geq\frac{1}{2}|S_{\gm_1}-S_{\gm_2}|.\]
\end{lemma} 
\begin{proof} If $\gm'=(\gm,x)\#\gm_{\mathbb{T}}$, then $|S_{\gm}-S_{\gm'}|$ is equal to $0$ or $1$ depending on whether $x$ is a side of a $1$-simple curve or not.\\  
On the other side, if $\gm'=\sigma_{x,y}(\gm)$ then $|S_{\gm}-S_{\gm'}|\leq 2$, that is a surgery creates at most two $1$-simple curves or eliminates at most two 1-simple curves. 

Since a path in $K_{\infty}$ is a sequence of surgery and connected sum, then we need at least $\frac{1}{2}|S_{\gm}-S_{\gm'}|$ steps from $\gm_1$ to $\gm_2$. 
\end{proof} 

Let $A:=\gm_{\mathbb{T}}$ and $X_{2g}$ be the unicellular collection obtained by gluing $g$-copies of $\gm_{\mathbb{T}}$ on the necklace $N_g$, each copy being glue on a 1-simple curve of $C_g$ (see Figure ~\ref{oiseau}). The collection $X_{2g}$ has $2g$ $1$-simple curves.

Let $Y_{2g}$ be the unicellular collection on $\sg_{2g}$ obtained by gluing $g-1$ copies of $\gm_{\mathbb{T}}$ to the necklace $N_{g+1}$ as in figure \ref{oiseau}.

\begin{figure}[htbp]
\begin{center}
\includegraphics[scale=0.25]{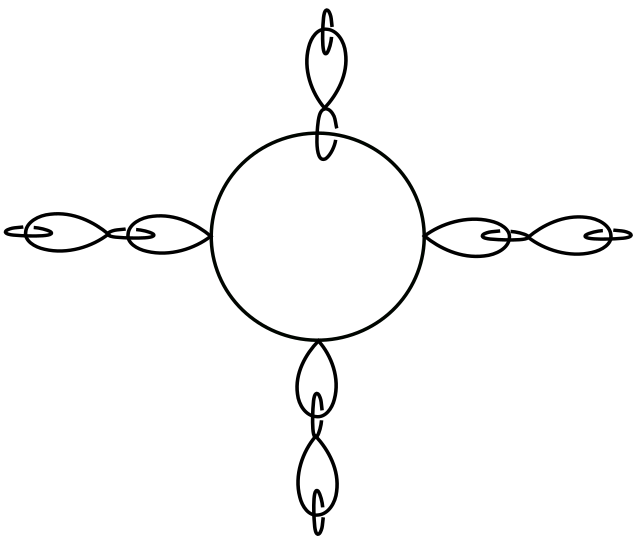}\hspace{2cm}
\includegraphics[scale=0.25]{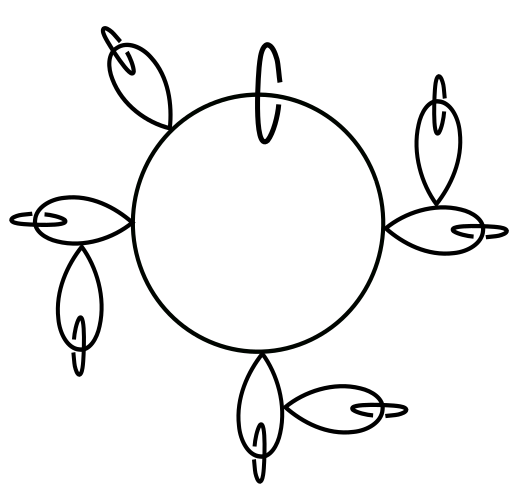}
\caption{The unicellular collections $X_{8}$ (on the left) and $Y_8$ (on the right).}
\label{oiseau}
\end{center}
\end{figure}

\begin{lemma}\label{depress}
For every $g\geq1$, $d(X_{2g}, \gm_{\mathbb{T}})=d(Y_{2g}, \gm_{\mathbb{T}})=2g$. Moreover,
\[\frac{g}{2} \leq d(X_{2g},Y_{2g})\leq2g.\]
\end{lemma}
\begin{proof} The collections $X_{2g}$ and $Y_{2g}$ are in the $2g$-th level of $K_{\infty}$, and are obtained by gluing $2g$ copies de $\gm_{\mathbb{T}}$. Therefore, $d(X_{2g}, \gm_{\mathbb{T}})=d(Y_{2g}, \gm_{\mathbb{T}})=2g$.

On $X_2g$, we cut the copies of $\gm_{\mathbb{T}}$ gluing on 1-simple curves and and glue them again in an apropriate manner to obtain $Y_{2g}$. Doing so, we obtained a path on $K_{\infty}$, from $X_{2g}$ to $Y_{2g}$ of length $2g$. Therefore, $$d(X_{2g},Y_{2g})\leq2g.$$  

On the other side, we have $|S_{X_{2g}}-S_{Y_{2g}}|=2g$, so $d(X_{2g},Y_{2g})\geq2g$.
\end{proof}

Let $T_{2g}$ be a triangle with extremities $A$, $X_{2g}$ and $Y_{2g}$. The points $X_{2g}$ and $Y_{2g}$ are in the same level $K_{4,2g}$, but we do not know whether a geodesic from $X_{2g}$ to $Y_{2g}$ stays in $K_{4,2g}$ or not. The level $K_{4,2g}$ is maybe not geodesic. Nonetheless, Lemma ~\ref{depress} tell us that the geodesic $(X_{2g}Y_{2g})$ do not go down the level $K_{4,g}$, that is $$d(\gm_{\mathbb{T}},(X_{2g}Y_{2g})).$$

In fact, since $X_{2g}$ and $Y_{2g}$ are in the same level, if the geodesic $(X_{2g}Y_{2g})$ go down in level $k$ times, it must go up in level $k$ times and it implies that $$2k\leq d(X_{2g},Y_{2g})\leq 2g \implies k\leq g.$$

This fact on $T_{2g}$ is crucial and it allows us to show that the sequence of triangles $(T_{2g})_{g\in\n}$ is not $\delta$-thin for any $\delta\geq 0$.

\begin{proof}[Proof of Theorem \ref{hyper}] Let $\mathcal{D}_{2k}:=(X_2k,\rightarrow)$ (respectively $\mathcal{D'}_{2k}:=(Y_2k,\rightarrow)$) the half-geodesic passing through all the points $X_{2m}$ (respectively $Y_{2m}$) for $m\geq k$.

Since $d(X_{2g},Y_{2g})\geq \frac{g}{2}$, then $d(\mathcal{D}_k,\mathcal{D'}_k) \rightarrow +\infty$. So, the $\delta$-neighborhoods $V_{\delta}(\mathcal{D}_k)$ and $V_{\delta}(\mathcal{D}'_k)$ are disjoint for $k$ sufficiently large and 
$$d(V_{\delta}(\mathcal{D}_k),V_{\delta}(\mathcal{D}'_k))\rightarrow +\infty.$$ 

It follows that for $k_0$ big enough, \[d(\gm_{\mathbb{T}},(X_{4k_0}Y_{4k_0}))>2k_0,\quad d(V_{\delta}(\mathcal{D}_k),V_{\delta}(\mathcal{D}'_k))\geq 1.\]

The geodesic $(X_{4k_0},Y_{4k_0})$ is not contained in $V_{\delta}(\mathcal{D}_k)\cup V_{\delta}(\mathcal{D}'_k)$. \\So, $(X_{4k_0}Y_{4k_0})$ is not contained in $V_{\delta}(\gm_{\mathbb{T}}X_{4k_0})\cup V_{\delta}(\gm_{\mathbb{T}}Y_{4k_0})$.

Hence, $K_{\infty}$ is not Gromov hyperbolic.  
\end{proof} 
\end{paragraph}  
\end{section}
\vspace{1,8cm}
\begin{question} The characterization of a surgery on a one-faced (Lemma ~\ref{lemsurg}) collection still holds for the general case, namely for unicellular maps. Therefore, one can wonder whether the surgery graph of unicellular maps is connected. 
\end{question}

A surgery on a unicellular map $\gm$ leaves the degree partition of $\gm$ invariant. Therefore, the surgery graph for unicellular maps is for unicellular maps with the same degree partition. 

Among one-faced collections, there is a big class of those made by only simple curves. We know that their number grows exponentially with the genus \cite{Aoug}.
\begin{question} Is the surgery graph of those one-faced collection made by simple curves connected? Is the surgery graph of minimally intersecting pairs connected?
\end{question}
In the first case, surgeries are allowed only between intertwined  oriented edges belonging to the same curve or to two disjoint curves. For minimally intersecting filling pairs, surgery are allowed only between intertwined  oriented edges on different sides of the same curve.

Using the Goupil-Schaeffer formula for the number of one-faced collections, one can show that $D_g$ is (asymptotically) at least linear on $g$.   
\begin{question} Is the diameter $D_g$ linear on $g$? Is the family $(K_g)$ an expander? 
\end{question}

Unité de Mathématiques Pures et Appliquées (UMPA), ENS-Lyon.\\
\textit{E-mail address}: abdoul-karim.sane@ens-lyon.fr


\begin{thebibliography}{ll}
\addcontentsline{toc}{chapter}{Bibliographie}
\bibitem{Aoug}
T. Aougab and S. Huang , \emph{Minimally intersecting filling pairs on surfaces}. Algebr. Geom. Topol., 15(2):903-932, 2015.

\bibitem{Chap}
G. Chapuy, \emph{A new combinatorial identity for unicellular maps, via a direct bijective approach.} Advances in Applied Mathematics, 47(4):874-893, 2011.

\bibitem{CFF}
G. Chapuy, V. Féray , E. Fusy, \emph{A simple model of trees for unicellular maps}, arXiv:1604.06688.

\bibitem{Ghys}
E. Ghys, P. de la Harpe, \emph{Sur les groupes hyperboliques d'après Mikhael Gromov.} Birkhäuser, 1990.

\bibitem{Goup}
A. Goupil and G. Schaeffer, \emph{Factoring n-cycles and counting maps of a given genus,} European J. Combin., 19(7): 819-834, 1998.

\bibitem{Zag}
J. Harer and D. Zagier, \emph{The Euler characteristic of the moduli space of curves,} Invent. Math., 85(3): 457-485, 1986.

\bibitem{Zvk}
S. K. Lando and A. K. Zvonkin, \emph{Graphs on surfaces and their applications.} Springer, 2004.

\bibitem{Element}
A. Sane, \emph{Intersection norm and one-faced collections,} arXiv:1809.03190.

\bibitem{Tut1}
W. T. Tutte, \emph{A census of Hamiltonian polygons,} Canad. J. Math., 14:402-417, 1962.
\bibitem{Tut2}
W. T. Tutte, \emph{A census of planar triangulations,} Canad. J. Math., 14:21-38, 1962.
\bibitem{Tut3}
W. T. Tutte, \emph{A census of slicing} Canad. J. Math., 14:708-722, 1962.
\bibitem{Tut4}
W. T. Tutte, \emph{A census planar graph,} Canad. J. Math., 15:249-271, 1963.

\bibitem{Walsh}
T. R. S. Walsh and A. B. Lehman, \emph{Counting rooted maps by genus.} I. J. Combin. Theory Ser. B, 13:192-218, 1972.

\end{thebibliography}
\end{document}